\documentclass[a4paper, 10pt]{article}

\title{Ramsey's theorem for pairs, \\ collection, and proof size}
\author{Leszek Aleksander Ko\l odziejczyk\\
 Institute of Mathematics\\
 University of Warsaw\\
 Banacha~2\\
 02-097 Warszawa, Poland\\
 E-mail: \sf{lak@mimuw.edu.pl}\and
Tin Lok Wong\\
 Department of Mathematics\\
 National University of Singapore\\
 10~Lower Kent Ridge Road\\
 Singapore 119076\\
 E-mail: \sf{matwong@nus.edu.sg}\and
Keita Yokoyama\footnote{corresponding author}\\
 School of Information Science\\
 Japan Advanced Institute of Science and Technology\\
 1-1 Asahidai\\
 Nomi, Ishikawa 923-1292, Japan\\
 E-mail:  \sf{y-keita@jaist.ac.jp}}
\date{}

\usepackage{amsmath}
\usepackage{amsthm}
\usepackage{amssymb}
\usepackage{mathrsfs}
\usepackage{stmaryrd}

\usepackage{url}

\theoremstyle{definition}
\newtheorem{thm}{Theorem}[section]
\newtheorem{defn}[thm]{Definition}
\newtheorem{lem}[thm]{Lemma}
\newtheorem{prop}[thm]{Proposition}
\newtheorem{cor}[thm]{Corollary}

\newtheorem{eg}[thm]{Example}

\theoremstyle{remark}
\newtheorem*{remark}{Remark}

\newtheorem*{claim*}{Claim}

\makeatletter
\newlength\@templen

\newcommand{\fitwd}[3][c]{%
 \settowidth{\@templen}{#3}%
 \makebox[\@templen][#1]{#2\vphantom{#3}}%
}
\long\def\ifblank#1{%
 \@ifblank#1\@sep\@sep\@secondoftwo\@firstoftwo\@sep\@sep
}
\long\def\@ifblank#1#2\@sep#3#4#5\@sep\@sep{#4}

\makeatother

\renewcommand{\exp}{\mathnoun{exp}}
\let\leq\leqslant
\let\geq\geqslant
\let\le\leq
\let\ge\geq
\let\phi\varphi
\renewcommand{\today}{\number\day\nobreakspace\ifcase\month\or
  January\or February\or March\or April\or May\or June\or
  July\or August\or September\or October\or November\or December\fi,
  \number\year}

\let\partref\eqref

\let\mathnoun\mathrm

\newcommand{\defm}[1]{\emph{#1}}
\newcommand{\hyp}{\text{-}}
\newcommand{\ldquot}{\mathopen\text{``}}
\newcommand{\rdquot}{\mathclose\text{''}}
\newcommand{\dquot}[1]{{\ldquot#1\rdquot}}

\newcommand{\qand}{\quad\text{and}\quad}

\newcommand{\defeq}{\mathrel{\mathop:}\mathrel{\mkern-1.2mu}=}

\newcommand{\parto}{\rightharpoonup}
\newcommand{\IN}{\mathbb{N}}

\newcommand{\lang}{\mathcal{L}}
\newcommand{\proves}{\vdash}
\newcommand{\then}{\rightarrow}

\newcommand{\nsc}{\leftrightarrow}

\DeclareMathOperator*{\bigwwedge}
 {\mathchoice{\bigwedge\mkern-16.5mu\bigwedge}
             {\bigwedge\mkern-12mu\bigwedge}
             {\bigwedge\mkern-13mu\bigwedge}
             {\bigwedge\mkern-11mu\bigwedge}}

\newcommand{\elemsub}{\preccurlyeq}
\newcommand{\elemext}{\succcurlyeq}
\newcommand{\elemsubneq}{\precneqq}

\newcommand{\Def}{\mathnoun{Def}}

\newcommand{\fa}[2]{\forall{#1}\, {#2}}
\newcommand{\ex}[2]{\exists{#1}\, {#2}}
\newcommand{\xfa}[1]{\forall{#1}\, }

\newcommand{\exi}[1]{\ex{{!}#1}}
\newcommand{\fain}[3]{\forall{#1{\in}{#2}}\, {#3}}
\newcommand{\exin}[3]{\exists{#1{\in}{#2}}\, {#3}}
\newcommand{\xfain}[2]{\forall{#1{\in}{#2}}\, }

\newcommand{\forces}{\Vdash}
\newcommand{\nforces}{\nVdash}
\newcommand{\defd}{\mathclose\downarrow}
\newcommand{\cname}[1]{\check{#1}}

\newcommand{\extendseq}{\trianglelefteqslant}

\newcommand{\faexteq}[4][]{%
 \ifblank{#1}
  {\forall{#2{\extendseq}#3}\ {#4}}
  {\forall{#2{\extendseq_{#1}}#3}\ {#4}}%
}
\newcommand{\exexteq}[4][]{%
 \ifblank{#1}
  {\exists{#2{\extendseq}#3}\ {#4}}
  {\exists{#2{\extendseq_{#1}}#3}\ {#4}}%
}

\newcommand{\Cond}[1][]{\ifblank{#1}{\mathnoun{Cond}}{\mathnoun{Cond}_{#1}}}

\newcommand{\ee}{\mathnoun{e}}

\newcommand{\eextneq}{\supsetneq_{\ee}}
\newcommand{\tuple}[1]{\langle#1\rangle}
\newcommand{\nmrl}[1]{\underline{#1}}

\newcommand{\Lone}{\lang_{1}}
\newcommand{\Ltwo}{\lang_{2}}
\newcommand{\X}{\mathscr{X}}

\newcommand{\ind}{\mathnoun{I}}
\newcommand{\bd}{\mathnoun{B}}
\newcommand{\RCA}{\mathnoun{RCA}}
\newcommand{\ACA}{\mathnoun{ACA}}
\newcommand{\WKL}{\mathnoun{WKL}}
\newcommand{\RT}{\mathnoun{RT}}
\newcommand{\PV}{\mathnoun{PV}}

\newcommand{\Con}{\mathnoun{Con}}
\newcommand{\Sat}{\mathnoun{Sat}}
\newcommand{\Cod}{\mathnoun{Cod}}

\newcommand{\falt}[3]{\forall{#1{<}{#2}}\, {#3}}

\newcommand{\fale}[3]{\forall{#1{\leq}{#2}}\, {#3}}
\newcommand{\exle}[3]{\exists{#1{\leq}{#2}}\, {#3}}
\newcommand{\xfale}[2]{\forall{#1{\leq}{#2}}\,}

\newcommand{\excf}[2]{\exists^\infty{#1}\, {#2}}
\newcommand{\excfin}[2]{\excf{#1{\in}#2}}
\newcommand{\aall}[2]{\forall^\infty{#1}\, {#2}}
\newcommand{\aallin}[2]{\aall{#1{\in}#2}}

\newcommand{\I}{\mathbb{I}}

\newcommand{\Ext}{\mathnoun{Ext}}

\newcommand{\Wx}[2][n+1]{\Ext_{#1}(#2)}
\newcommand{\val}[3][n+1]{#2(#3)}

\newcommand{\forcesp}{\Vdash^{+}}

\newcommand{\VF}{\mathscr{V}}

\newcommand\fin{\mathrm{fin}}

\usepackage{todonotes}
\usepackage[shortlabels]{enumitem}

\newcommand{\changed}[1]{\textcolor{blue}{#1}}

\begin{document}
\maketitle

\begin{abstract}
We prove that any proof of a $\forall \Sigma^0_2$ sentence in the theory $\WKL_0 + \RT^2_2$
can be translated into a proof in $\RCA_0$ at the cost of a polynomial increase in size.
In fact, the proof in $\RCA_0$ can be obtained by a polynomial-time algorithm.
On the other hand, $\RT^2_2$ has non-elementary speedup over the weaker base theory $\RCA^*_0$
for proofs of $\Sigma_1$ sentences.

We also show that for $n \ge 0$, proofs of $\Pi_{n+2}$ sentences in $\bd \Sigma_{n+1}~+~\exp$
can be translated into proofs in  $\ind \Sigma_{n} + \exp$ at a polynomial cost in size. Moreover,
the $\Pi_{n+2}$-conservativity of $\bd \Sigma_{n+1} + \exp$ over $\ind \Sigma_{n} + \exp$
can be proved in $\PV$, a fragment of bounded arithmetic corresponding to polynomial-time computation.
For $n \ge 1$, this answers a question of Clote, H\'ajek, and Paris.


\noindent
{\bf MSC classes} Primary: 03F20, 03B30, 03F35 Secondary: 05D10, 03F30, 03F25, 03H15
\end{abstract}


%
%

\renewcommand{\changed}{}

\changed{
The logical strength of Ramsey's theorem for pairs and two colours, 
formalized as the second-order arithmetic statement $\RT^2_2$,
has been a major topic of interest in reverse mathematics and related areas of logic 
for over 25 years \cite{Seetapun1995strength, CJS01, CSY2014, CSY2017}.
Recently, Patey and the third author \cite{PY} showed that,
as far as proving relatively simple statements is concerned, 
$\RT^2_2$ is no stronger than the usual base theory considered in reverse mathematics,
$\RCA_0$, axiomatized by the comprehension principle for computable properties of numbers and by $\Sigma^0_1$ induction.
More precisely, the result shown in \cite{PY} is that
$\RT^2_2$ is \emph{$\forall \Sigma^0_2$-conservative} over $\RCA_0$: 
for any sentence $\gamma$ that consists of 
a block of universal first- and second-order quantifiers followed by a $\Sigma^0_2$ formula, 
if $\RCA_0 + \RT^2_2$ proves $\gamma$, then $\gamma$ already has a proof in $\RCA_0$.  
In particular, since any statement expressing the well-foundedness of a given computable
ordering is $\forall\Sigma^0_2$, the theories $\RCA_0 + \RT^2_2$ and $\RCA_0$ have the same
proof-theoretic strength as measured in terms of provability of the well-foundedness 
of ordinal notation systems.
}

\changed{
This new insight into the strength of $\RT^2_2$ naturally leads to some further questions.
One of these is whether the conservation result can be extended to a complete characterization
of the $\Pi^1_1$ sentences provable from $\RT^2_2$. For example, it is open whether the
$\Pi^1_1$ consequences of $\RCA_0 + \RT^2_2$ are axiomatized by the $\Sigma^0_2$ collection
principle (which itself witnesses that $\RT^2_2$ is not $\Pi^1_1$-conservative over $\RCA_0$).
Here, we consider a different problem, already raised as Question 9.5 in \cite{PY}:
does $\RT^2_2$ have significant proof speedup over $\RCA_0$ with respect to $\forall\Sigma^0_2$ sentences?
A positive answer would indicate, among other things, 
that the conservation result might have rather limited relevance in practice.  
}

The study of proof size in axiomatic theories \changed{and of proof speedup},
a topic going back to G\"odel \cite{goedel:laenge},
was given an excellent (if no longer fully up-to-date) survey by Pudl\'ak~\cite{incoll:pflen}.
An important phenomenon that has been observed empirically is that, from a quantitative perspective,
prominent cases in which an arithmetic theory 
\changed{$T$ is conservative over a theory $S$ for sentences in some class $\Gamma$
typically fit one of two patterns:
either $T$ has iterated exponential (``tower function'') speedup over $S$ on proofs of sentences from $\Gamma$,
or each proof of such a sentence in $T$ can be translated into $S$ with at most polynomial blowup.}
The former behaviour is illustrated for instance by the conservativity of arithmetical comprehension $\ACA_0$ over $\mathnoun{PA}$
and the $\Pi_2$ conservativity of $\ind\Sigma_1$ over primitive recursive arithmetic.
The latter is illustrated by the conservativity of $\RCA_0$ over $\ind\Sigma_1$ and the $\Pi^1_1$-conservativity of Weak K\"{o}nig's Lemma over $\RCA_0$.
The question we are interested in is whether the conservativity result of $\cite{PY}$ also fits one of these patterns and, if so, which one.

Our main result is that 
\changed{the conservativity theorem of \cite{PY} is, quantitatively speaking, tame:}
$\RT^2_2$ has at most polynomial speedup over $\RCA_0$ for $\forall\Sigma^0_2$ sentences.
Moreover, this is witnessed by a \emph{polynomial simulation}:
there is a polytime procedure that takes a proof of a $\forall\Sigma^0_2$ sentence
in $\WKL_0 + \RT^2_2$ (where $\WKL_0$ is $\RCA_0$ extended by Weak K\"{o}nig's Lemma)
and outputs a proof of the same sentence in $\RCA_0$.
By \cite[Proposition 4.6]{ci:instrumentalism}, it follows that proofs of \changed{purely first-order} 
$\Pi_3$ sentences in $\WKL_0 + \RT^2_2$
can be translated in polynomial time into proofs in $\ind\Sigma_1$.

To obtain our result, we make use of a general-purpose technique due to Avigad \cite{art:formforce}:
in order to show polynomial simulation of a theory 
\changed{$T$ by $S$, build what could be called a ``forcing interpretation'' of $T$ in $S$ ---
that is, formalize within $S$ a forcing construction that leads to a model of $T$ ---
and verify via small proofs in $S$}
that sentences of the appropriate class are forced if they are true in the ground model.
Interestingly, while most forcing arguments give rise to a generic extension of the ground model,
the one we work with produces a \emph{generic cut}, refining an initial segment construction used in \cite{PY}.
Both the construction of \cite{PY} and ours rely on non-trivial finite combinatorics.
In the case of \cite{PY}, this took the form of an upper bound on 
\changed{finite Ramsey's theorem with the size of finite sets expressed in terms of so-called $\alpha$-largeness.}
Here, we need a considerable strengthening of that bound with a more constructive proof, which was provided in \cite{KY:ramsey-combinatorics}.

We also consider the question whether an analogous polynomial simulation result
still holds if we weaken the base theory to $\RCA^*_0$, which differs from
$\RCA_0$ in that the $\Sigma^0_1$ induction scheme is replaced by
$\Delta^0_0$ induction plus an axiom $\exp$ guaranteeing totality of the exponential function.
It turns out that, perhaps surprisingly, changing the base theory makes a major difference:
even though $\RT^2_2$ is $\forall \Sigma^0_2$-conservative over $\RCA^*_0$
(as can be shown by the method used to derive $\Pi^0_2$-conservativity in \cite{Y-MLQ13}),
it has iterated exponential speedup over $\RCA^*_0$, already for 
\changed{sentences of very limited syntactic complexity.
We give a proof of the speedup for $\Sigma_1$ sentences; 
in fact, it could also be witnessed by finite consistency statements 
that can are expressible by bounded formulas with exponential terms.}

Since $\RCA_0 + \RT^2_2$ proves $\Sigma^0_2$ collection, our main theorem immediately
implies that the \changed{so-called} Paris--Friedman theorem, i.e.~the $\Pi_{n+2}$-conservativity of $\bd \Sigma_{n+1}$ over $\ind\Sigma_n$,
can be strengthened to a polynomial simulation in the case when $n=1$. It makes sense to ask whether this extends to other values
of $n$, especially because a related question --- whether the conservativity can be proved in bounded arithmetic ---
was asked by Clote~et~al.~\cite{art:formconserv}. We prove that the answer to both questions is positive for all $n \ge 1$,
and that it is positive for $n =0$ if both $\bd \Sigma_1$ and $\ind \Delta_0$ are extended by $\exp$.
We do this by turning one of the model-theoretic proofs of Paris--Friedman into a forcing interpretation.
As discussed in slightly greater detail in Section 4, Fedor Pakhomov [private communication] independently devised a completely different approach
that proves a more general result; in particular, his argument obviates the need for $\exp$ in the $n = 0$ case.

The remainder of this paper has \changed{the following} structure.
After introducing some basic definitions and notational conventions below,
we discuss the general concept of forcing interpretations in Section 1.
We prove our main theorem in Section 2 and the contrasting speedup result over $\RCA^*_0$ in Section~3.
The polynomial simulation of $\bd \Sigma_{n+1} + \exp$ by $\ind\Sigma_n + \exp$ is presented in Section~4.

\begin{center}
*
\end{center}

\changed{Basic information on fragments of first- and second-order arithmetic
can be found in \cite{book:hajek+pudlak}, \cite{SOSOA}, and \cite{hirschfeldt:slicing-book}.}

We will write~\defm{$\Lone$} and~\defm{$\Ltwo$} for the languages of first- and second-order arithmetic, respectively.
Note that these languages do \emph{not} have a symbol for exponentiation.
We use lowercase letters for objects of the numerical (``first-order'') sort,
and uppercase letters for objects of the set (``second-order'') sort.
Notation like $\Sigma^0_n$, $\Pi^0_n$ represents the usual formula classes defined
in terms of first-order quantifier alternations, but allowing second-order free variables.
On the other hand, notation without the superscript $0$, like $\Sigma_n$, $\Pi_n$,
represents analogously defined
classes of $\Lone$ formulas --- that is, without any second-order variables at all.
If we want to specify the second-order parameters appearing
in a $\Sigma^0_n$ formula, we use notation like $\Sigma_n(\bar X)$.
We extend these conventions to naming theories: thus, for example,
$\bd\Sigma^0_2$ is the $\Ltwo$-theory axiomatized by $\Delta^0_0$ induction and
$\Sigma^0_2$ collection, whereas $\bd\Sigma_2$ is the $\Lone$-theory
axiomatized by $\Delta_0$ induction and $\Sigma_2$ collection.

Recall that $\RCA^*_0$ is the theory defined in \cite{simpson-smith} which differs
from $\RCA_0$ in that $\Sigma^0_1$ induction is replaced by $\Delta^0_0$
induction plus an axiom\changed{, known as $\exp$,} stating the totality of exponentiation.
The first-order consequences of $\RCA^*_0$ are known to coincide with the theory $\bd\Sigma_1 + \exp$,
which in particular means that $\RCA^*_0$ is $\Pi_2$-conservative over $\ind\Delta_0 + \exp$.

The symbol $\mathbb{N}$ stands for the set of natural numbers ---
both as understood in the metatheory (the \emph{standard} natural numbers)
or as formalized in an $\Lone$ or $\Ltwo$ theory
(which, if consistent, will of course have nonstandard models).
It should be clear from the context which is meant.
A set $X \subseteq \IN$ is \emph{finite}
if it is bounded, i.e.~there is $k \in \IN$ such that $\ell \le k$ for all $\ell \in X$;
otherwise, $X$ is \emph{infinite}. Write $X\subseteq_{\fin}Y$ for $X$ is a finite subset of $Y$. Each finite set $X$ is coded in a standard way by
the binary representation of some natural number $x$,
and may be identified with $x$ in many contexts.
The symbol $\omega$ stands for the smallest infinite ordinal --- again, we use
the same symbol in the metatheory and in formal theories.


We will be interested in comparing the sizes of proofs of statements
in various theories, which requires us to fix some terminology and
conventions regarding syntax. We use the word \defm{theory} to mean a set of sentences.
Every theory~$T$ comes with a language~\defm{$\lang(T)$}
with the property that each non-logical symbol appearing in $T$~is in $\lang(T)$,
though not necessarily vice versa. To simplify things, \emph{we assume that
$\lang(T)$ always contains only finitely many non-logical symbol}s.
The \defm{size} of a term, a formula or a proof is the number of symbols in it.
All variable, constant, function and relation symbols count as one symbol.
We use vertical lines, $|\cdot|$, to denote the size of a given syntactical object.

To measure the sizes of proofs precisely, one needs to fix a proof system.
It is known that there are polynomial-time translations
between the usual Hilbert-style systems and the sequent calculus \emph{with the cut rule} \cite{eder:complexities-calculi},
and between their tree-like or sequence-like versions \cite[Theorem~4.1]{incoll:pflen}.
Therefore, in principle, it does not matter which proof system we choose here.
In our arguments, we have in mind a Hilbert-style system
like the one in Enderton~\cite[Section~2.4]{book:enderton},
except that instead of allowing all (universal closures of) propositional tautologies as axioms,
we derive them from a finite number of propositional axiom schemes.
We assume that, like in Enderton's system, the only official connectives are $\neg, \then$ and the only quantifier is $\forall$.
Recall~\cite[page~113]{book:enderton} that a term~$t$ is \defm{substitutable} for a variable symbol~$v$
in a formula~$\phi$ if no free occurrence of~$v$ in~$\phi$ is in the scope of a quantifier binding
a variable that appears in~$t$.

\begin{defn}
Let $T,T'$ be theories and let $\Gamma$ be a set of sentences in the language~$\lang(T)\cap\lang(T')$.
Then $T$ \emph{polynomially simulates}~$T'$ with respect to~$\Gamma$
if there exists a polynomial-time procedure that, given any proof of $\gamma \in \Gamma$ in $T'$ as input,
outputs a proof of $\gamma$ in $T'$. (In particular, this implies that for every proof of $\gamma \in \Gamma$
in $T'$, there is a proof of $\gamma$ in $T$ of at most polynomially larger size.)

The theory $T'$ has \emph{non-elementary speedup} over $T$ with respect to $\Gamma$ if for each elementary recursive
function $f$, there exist $\gamma\in\Gamma$ and a proof $\pi$ of~$\gamma$ in~$T'$
such that no proof of~$\gamma$ in~$T$ has size $\le f(|\pi|)$.
\end{defn}

\section{Forcing interpretations}\label{sec:forcing}

The technique we use for proving polynomial simulations was developed in Avigad~\cite{art:formforce}
and relies on a certain kind of formalized forcing argument.
In essence, the idea is to define a more general notion of interpretation:
while a traditional interpretation gives a uniform way of defining a model $M' \models T'$
inside a model $M \models T$, an Avigad-style interpretation gives a way of describing
a generic model $M[G] \models T'$ which is not fully specified until the generic filter $G$ is fixed. The point is that,
even though $G$ might not be definable, the properties of $M[G]$ that we care about in
the context at hand do not depend on $G$.

In this section, we discuss this idea, and its connection to questions of proof size, in some generality.
We begin with an annoying technical issue.

\subsection{Simplification}
When defining intepretations, formulas with more than one function symbol in an atomic subformula
can cause ambiguities. In this subsection, we explain how to get around the problem
by avoiding such formulas as much as possible.
This is quite similar in spirit to translating all formulas to a relational language.

\begin{defn}
A \defm{simple} term is one of the form $v$ or $f(\bar w)$,
 where $v,\bar w$ are variables and $f$ is a function symbol.
A \defm{simple} formula is a formula
 in which every atomic subformula is either
 \begin{itemize}
 \item a formula with no function symbol; or
 \item an equation with exactly one function symbol.
 \end{itemize}
Constant symbols are regarded as $0$-ary function symbols here.
\end{defn}

\begin{defn}\label{defn:simple-tr}
For a formula $\theta$, its simple translation~$\theta^*$ is defined in the following way.

First, to each term $t$ and each variable $x$ not appearing in $t$,
 we associate a simple formula $t[x]$,
   intended to define $x$ to be the term $t$,
  by induction on the subterms of $t$.
For a variable $y$, the formula $y[x]$ is $x = y$. For
$t$ of the form $f(t_1,\ldots,t_n)$, the formula $t[x]$ is
 \begin{equation*}
 \forall x_1 \ldots \forall x_n\,(t_1[x_1] \then (t_2[x_2] \then \ldots \then (t_n[x_n] \then x = f(x_1,\ldots,x_n))\ldots )),
 \end{equation*}
 where $x_1,\dots,x_n$ are canonically chosen fresh variables.

The formula $\theta^*$ equals $\theta$ for $\theta$ simple atomic. If $\theta$ is
$R(t_1,\ldots,t_n)$ and is not simple, then $\theta^*$ is \[\forall x_1\ldots \forall x_n\,(t_1[x_1] \then (t_2[x_2] \then \ldots \then (t_n[x_n] \then R(x_1,\ldots,x_n))\ldots ))\] where $x_1,\ldots,x_n$ are canonically chosen fresh variables. The translation
commutes with $\neg$, $\then$, and $\forall$.
\end{defn}

The following two lemmas mean that when studying questions of proof size,
we can largely restrict attention to simple formulas.
The lemmas can essentially be read off Visser's exposition in~\cite[Section~7.3]{art:insideEXP}.

\begin{lem}\label{lem:fma-rel}
There exists a polynomial-time procedure that, given a formula $\theta$, outputs
a first-order logic proof of~$\theta\nsc\theta^*$.
\end{lem}

\begin{proof}
This is essentially Theorem~7.3.6 in Visser~\cite{art:insideEXP}, with the following differences:
(i) Visser defines the translation $\theta \mapsto \theta^*$ (where $\theta^*$ would be $\theta^{*\circ}$ in his notation)
using $\exists$ and $\land$, while we use $\forall$ and $\then$ to match the official connectives of our proof system,
and (ii) Visser only needs a formal proof of~$\theta\nsc\theta^*$ of size polynomial in $|\theta|$,
so he does not discuss the time required to construct it. However, these differences have no bearing on the proof,
and the reader can verify that all of Visser's constructions can be carried out in polynomial time.
\end{proof}

\begin{lem}\label{lem:pf-rel}
There is a polynomial-time procedure that, given a proof $\pi$ in first-order logic,
outputs a proof $\pi^*$ with the following properties:
 \begin{itemize}
 \item $\pi^*$ is a proof in the same language as~$\pi$;
 \item every formula in~$\pi^*$ is simple;
 \item if $\theta_1,\theta_2,\dots,\theta_\ell$
    are the non-logical axioms in~$\pi$,
   then $\theta_1^*,\theta_2^*,\dots,\theta_\ell^*$
    are the non-logical axioms in~$\pi^*$;
 \item if $\eta$~is the conclusion of~$\pi$,
   then $\eta^*$~is the conclusion of~$\pi^*$.
 \end{itemize}
\end{lem}

\begin{proof}
See Theorem~7.3.3 and Theorem~7.3.4 in Visser~\cite{art:insideEXP},
with the same caveats as in Lemma~\ref{lem:fma-rel}.
\end{proof}


\subsection{Forcing translations and interpretations}

A traditionally understood interpretation of a theory $T'$ in a theory $T$
is essentially (i) a translation of $\lang(T')$ into $\lang(T)$
such that (ii) $T$ proves the translations of axioms of $T'$. Forcing interpretations of the sort we need here
have a similar two-layered structure of ``translation'' and ``interpretation proper'',
except that now instead of translating an $\lang(T')$~formula $\varphi$ into $\lang(T)$,
we have to translate ``$\varphi$ is forced'' into $\lang(T)$. Then we have to verify in $T$
that the axioms of $T'$ are forced.

\begin{defn}
A \defm{forcing translation~$\tau$}
  from a language~$\lang'$ to a language~$\lang$
 consists of $\lang$~formulas
  \begin{equation*}
   s\in\Cond[\tau],\quad
   s'\extendseq_\tau s,\quad
   s\forces_\tau v\defd,\quad
   s\forces_\tau \alpha(v_1,v_2,\dots,v_\ell)
  \end{equation*}
  for every simple atomic $\lang'$~formula~$\alpha(v_1,v_2,\dots,v_\ell)$
   such that
   \begin{enumerate}[(FT1)]
   \item $s'\extendseq_\tau s$ contains
    $s'\in\Cond[\tau]\wedge s\in\Cond[\tau]$ as a conjunct;
   \item $s\forces_\tau v\defd$ contains $s\in\Cond[\tau]$ as a conjunct;
   \item $s\forces_\tau \alpha(v_1,v_2,\dots,v_\ell)$ contains\label{item:ftr/f>defd}
    $\bigwwedge_{i=1}^\ell s\forces_\tau v_i\defd$ as a conjunct,
     whenever $\alpha(v_1,v_2,\dots,v_\ell)$ is a simple atomic $\lang'$~formula; and
   \item
   if $\alpha(\bar u,v)$ is a simple atomic $\lang'$~formula and $w$~is a variable,
     then\label{item:ftr/varsubst}
     \begin{equation*}
      \bigl(s\forces_\tau\alpha(\bar u,v)\bigr)[v/w]
      \qand s\forces_\tau\bigl(\alpha(\bar u,v)[v/w]\bigr)
     \end{equation*} are the same.
   \end{enumerate}
All formulas above have exactly the free variables shown.
When there is no risk of ambiguity,
 we will often omit the subscript~$\tau$ in the notation.
We read `$s\forces\dots$' as `$s$~forces \ldots'.
By convention, the variable symbols $s,s',s'',\dots$
 are always distinct from $u,v,w,z,\dots$.
In accordance with the usual customs related to forcing,
we refer to elements~$s$ satisfying $s\in\Cond$ as forcing \defm{conditions},
and think of the objects denoted by the variables $u,v,w,z,\ldots$
as \emph{names} (for the elements of the generic model being described).

\begin{remark}
In some contexts, it might make sense to allow forcing conditions
to be tuples of elements instead of single elements.
However, in this paper we deal exclusively with theories that have a definable pairing function. So to avoid complicating the notation we eschew that kind of generality
and continue to write $s$ rather than $\bar s$.
\end{remark}

\end{defn}

The clauses we use to extend a forcing relation to arbitrary simple formulas
follow those of a ``good strong forcing notion''
in the sense of Avigad~\cite[Definition~4.2]{art:formforce}.
We deal with non-simple formulas using the $(\cdot)^*$ translation.
\begin{defn}
Let $\tau$ be a forcing translation
  from a language~$\lang'$ to a language~$\lang$.
We define an $\lang$~formula
 \defm{$s\forces_\tau \theta(v_1,v_2,\dots,v_\ell)$}
  for each $\lang'$~formula $\theta(v_1,v_2,\dots,v_\ell)$
   by recursion on~$\theta$ as follows.
\begin{enumerate}[(FT1)]
\addtocounter{enumi}{4}
\item If $\theta(v_1,v_2,\dots,v_\ell)$ is a simple $\lang'$~formula,
 then\label{item:ftr/neg}
  \begin{equation*}
   s\forces_\tau \neg\theta(v_1,v_2,\dots,v_\ell)
  \end{equation*}
  is defined to be
  \begin{equation*}
   \bigwwedge_{i=1}^\ell\dquot{s\forces_\tau v_i\defd}
   \wedge \faexteq[\tau]{s'}s{s'\nforces_\tau\theta(\bar v)}.
  \end{equation*}

\item If $\theta(u_1,u_2,\dots,u_k,w_1,w_2,\dots,w_m)$ and
         $\eta(v_1,v_2,\dots,v_\ell,w_1,w_2,\dots,w_m)$
  are simple $\lang'$~formulas,
 then\label{item:ftr/then}
  \begin{equation*}
   s\forces_\tau \theta(u_1,u_2,\dots,u_k,w_1,w_2,\dots,w_m)
                 \then\eta(v_1,v_2,\dots,v_\ell,w_1,w_2,\dots,w_m)
  \end{equation*}
  is defined to be
  \begin{align*}
   &\bigwwedge_{i=1}^k\dquot{s\forces_\tau u_i\defd}
    \wedge \bigwwedge_{i=1}^\ell\dquot{s\forces_\tau v_i\defd}
    \wedge \bigwwedge_{i=1}^m\dquot{s\forces_\tau w_i\defd}\\
   &\wedge\faexteq[\tau]{s'}s{\exexteq[\tau]{s''}{s'}{\bigl(
     \dquot{s'\forces_\tau \theta(\bar u,\bar w)}
     \then\dquot{s''\forces_\tau \eta(\bar v,\bar w)}
    \bigr)}}.
  \end{align*}

\item If $\theta(v_1,v_2,\dots,v_\ell,w)$ is a simple $\lang'$~formula,
 then\label{item:ftr/forall}
  \begin{equation*}
   s\forces_\tau \fa w{\theta(v_1,v_2,\dots,v_\ell,w)}
  \end{equation*}
  is defined to be
  \begin{equation*}
   \bigwwedge_{i=1}^\ell\dquot{s\forces_\tau v_i\defd}
   \wedge\fa w{
    \faexteq[\tau]{s'}s{\exexteq[\tau]{s''}{s'}{\bigl(
     \dquot{s'\forces_\tau w\defd}
     \then\dquot{s''\forces_\tau\theta(\bar v,w)}
    \bigr)}}
   }.
  \end{equation*}

\item\label{item:ftr/nonsimple} If $\theta(\bar v)$ is an $\lang'$~formula that is not simple,\label{item:ftr/nsimple}
 then $s\forces_\tau\theta(\bar v)$ is defined to be $s\forces_\tau\theta^*(\bar v)$,
  where $\theta^*(\bar v)$ is as defined in Definition~\ref{defn:simple-tr}.

\item $\wedge,\vee,\nsc,\exists$ are defined in terms of $\neg,\then,\forall$
 in the usual way.\label{item:ftr/defd}
\end{enumerate}
Here the quotation marks $\dquot\dots$ are simply another type of brackets
 to enhance the readability of formulas.
Our convention is that $\forces$~has lower precedence
 than all the logical connectives in first-order logic.
All formulas above have exactly the free variables shown.
We often abbreviate $\bigwwedge_i(s\forces_\tau v_i\defd)$
 as \defm{$s\forces_\tau\bar v\defd$}.
\end{defn}

The definitions above are designed to
 make \ref{item:ftr/f>defd} and~\ref{item:ftr/varsubst}
  hold for all formulas, at least essentially.

\begin{lem}\label{lem:f>defd}
Let $\tau$ be a forcing translation
from a language~$\lang'$ to a language~$\lang$.
Then:
\begin{enumerate}[(i)]
\item $s\forces_\tau\theta(\bar v)$ contains $s\forces\bar v\defd$ as a conjunct
 for every $\lang'$~formula $\theta(\bar v)$;
\item if $\theta(\bar u,v)$ is a $\lang'$~formula and $w$~is a variable,
 then
 \begin{equation*}
  \bigl(s\forces_\tau\theta(\bar u,v)\bigr)[v/w]
  \qand s\forces_\tau\bigl(\theta(\bar u,v)[v/w]\bigr)
 \end{equation*}
  differ only by a one-to-one renaming of bound variables, and
   provided $\theta$ is simple,
    the variable~$w$ is substitutable for~$v$ in~$\theta(\bar u,v)$
     exactly when it is so in $s\forces_\tau\theta(\bar u,v)$.
\end{enumerate}
\end{lem}

\begin{proof}
Part (i) is true by construction. Part (ii), which is a special case of
Lemma~4.3 in Avigad~\cite{art:formforce},
can be proved by a straightforward induction on~$\theta$ using the definitions
 when $\theta$ is simple.
Our convention that the variable symbols $s,s',s'',\dots$
 are always distinct from $u,v,w,z,\dots$ is used to
  establish the substitutability part.
The rest follows from~\ref{item:ftr/nonsimple}.
\end{proof}

\begin{defn}
A \defm{forcing interpretation} of a theory~$T'$ in a theory~$T$
 is a forcing translation~$\tau$ from $\lang(T')$ to~$\lang(T)$
  such that $T$~proves
 \begin{enumerate}[(F{I}1)]
 \addtocounter{enumi}{-1}
 \item $\ex s{(s\in\Cond[\tau])}$;\label{item:fint/ex_Cond}

 \item $\fain s{\Cond[\tau]}{s\extendseq_\tau s}$;\label{item:fint/refl}

 \item\label{item:fint/trans}
 \begin{math}
  \fain{s,s',s''}{\Cond[\tau]}{\bigl(
   s''\extendseq_\tau s'\wedge s'\extendseq_\tau s\then s''\extendseq_\tau s
  \bigr)}
 \end{math};

 \item\label{item:fint/ex_Name}
 \begin{math}
  \fain s{\Cond[\tau]}{\exexteq[\tau]{s'}s{\ex v{
   s'\forces_\tau v\defd
  }}};
 \end{math}

 \item\label{item:fint/monodefd}
 \begin{math}
  \fain{s,s'}{\Cond[\tau]}{\fa v{\bigl(
   s'\extendseq_\tau s\wedge\dquot{s\forces_\tau v\defd}
   \then\dquot{s'\forces_\tau v\defd}
  \bigr)}}
 \end{math};

 \item for all simple atomic $\lang(T')$~formulas $\alpha(\bar v)$,\label{item:fint/monotone}
 \begin{equation*}
  \fain{s,s'}{\Cond[\tau]}{\fa{\bar v}{\bigl(
   s'\extendseq_\tau s\wedge \dquot{s\forces_\tau\alpha(\bar v)}
   \then\dquot{s'\forces_\tau \alpha(\bar v)}
  \bigr)}};
 \end{equation*}

 \item\label{item:fint/=refl}
 \begin{math}
  \fain s{\Cond[\tau]}{\fa v{(
   \dquot{s\forces_\tau v\defd}\then\dquot{s\forces_\tau v=v}
  )}}
 \end{math};

 \item\label{item:fint/=sym}
 \begin{math}
  \fain s{\Cond[\tau]}{\fa{u,v}{(
   \dquot{s\forces_\tau u=v}\then\dquot{s\forces_\tau v=u}
  )}}
 \end{math};

 \item\label{item:fint/=trans}
 \begin{math}
  \fain s{\Cond[\tau]}{\fa{u,v,w}{
   (
    \dquot{s\forces_\tau u=v}\wedge\dquot{s\forces_\tau v=w}
    \then\dquot{s\forces_\tau u=w}
   )
  }}
 \end{math};

 \item\label{item:fint/functions} for all $k\in\IN$ and\label{item:fint/fn}
           all $k$-ary function symbols~$f$ in~$\lang(T')$,
 \begin{multline*}
  \xfain s{\Cond[\tau]} \xfa{v_1,v_2,\dots,v_k}
  \Bigl(
   \bigwwedge_{i=1}^k \dquot{s\forces_\tau v_i\defd}\\
   \then\begin{aligned}[t]
    &\faexteq{s'}s{\exexteq{s''}{s'}{\ex w{\dquot{s''\forces_\tau w=f(\bar v)}}}}\\
    &\wedge s\forces_\tau \fa{w,w'}{\bigl(
             w=f(\bar v)
             \then(w'=f(\bar v)\then w=w')
            \bigr)}
  \Bigr);
   \end{aligned}
 \end{multline*}

\item for all simple $\lang(T')$~terms~$t(v_1,v_2,\dots,v_k)$
           with exactly the free variables shown and
          all simple atomic $\lang(T')$~formulas $\alpha(\bar u,w_0)$
 such that $\alpha(\bar u,t(\bar v))$ is simple,\label{item:fint/stermsubst}
  \begin{multline*}
   \xfain s{\Cond[\tau]} \xfa{\bar u,v_1,v_2,\dots,v_k,w} \\
   \bigl(
    \dquot{s\forces_\tau w=t(\bar v)}
    \then(
     \dquot{s\forces_\tau\alpha(\bar u,w)}
     \nsc\dquot{s\forces_\tau\alpha(\bar u,t(\bar v))}
    )
   \bigr);
  \end{multline*}

 \item\label{item:fint/densedefd}
 \begin{math}
  \fain s{\Cond[\tau]}{\fa v{(
   \faexteq[\tau]{s'}s{\exexteq[\tau]{s''}{s'}{\dquot{s''\forces_\tau v\defd}}}
   \then s\forces_\tau v\defd
  )}}
 \end{math};

 \item for all simple atomic $\lang(T')$~formulas $\alpha(\bar v)$,\label{item:fint/density}
 \begin{equation*}
  \fain s{\Cond[\tau]}{\fa{\bar v}{\bigl(
   \faexteq[\tau]{s'}s{\exexteq[\tau]{s''}{s'}{\dquot{s''\forces_\tau \alpha(\bar v)}}}
   \then s\forces_\tau\alpha(\bar v)
  \bigr)}};
 \end{equation*}

 \item\label{item:fint/theory} $\fain s{\Cond[\tau]}{s\forces_\tau \sigma}$
   \quad for all $\sigma\in T'$.\label{item:fint/thy}
 \end{enumerate}
We refer to a forcing translation satisfying \ref{item:fint/ex_Cond}--\ref{item:fint/density} (that is, to a forcing interpretation
of pure logic) as a \emph{forcing interpretation of $\lang(T')$ in $T$}.
\end{defn}

Clauses~\ref{item:fint/ex_Cond} and~\ref{item:fint/ex_Name} are technical.
\mbox{\ref{item:fint/refl}--\ref{item:fint/trans}} state that $\extendseq$
is a preorder. Meanwhile, we require neither antisymmetry nor the existence of a maximal element.
\mbox{\ref{item:fint/monodefd}--\ref{item:fint/monotone}} form the base case of the usual requirement that forcing be closed downwards under $\extendseq$.
\mbox{\ref{item:fint/=refl}--\ref{item:fint/stermsubst}} say roughly that the equality axioms are forced.
The variable symbol~$w$ in~\ref{item:fint/stermsubst}
may be syntactically equal to one of~$\bar u$,
 but by convention $w_0$~cannot be.
\mbox{\ref{item:fint/densedefd}--\ref{item:fint/density}} express the base case of the
connection between forcing $\neg\neg \varphi$ and forcing $\varphi$.
Finally, the crucial condition \ref{item:fint/theory} is what
makes a forcing interpretation of pure logic
be an interpretation of a theory $T'$: all axioms of $T'$ are forced.

Forcing interpretations are clearly
 closed under composition and definition by cases.
As shown by the following example,
 our notion of forcing interpretations generalizes
  the usual notion of interpretations.

\begin{eg}\label{eg:interpret}
Every interpretation (or, more precisely, every parameter-free one-dimensional global relative interpretation)~$\tau$ 
of a theory~$T'$ in a theory~$T$
gives rise to a forcing interpretation of~$T'$ in~$T$ as follows.
\begin{itemize}
\item Define $s\in\Cond[\tau]$ and $s'\extendseq_\tau s$ to be respectively
 \begin{equation*}
  s=s\qand s'=s'\wedge s=s.
 \end{equation*}
\item Define $s\forces_\tau v\defd$
 to be $s\in\Cond[\tau]\wedge\delta(v)$,
  where $\delta(v)$ is the defining formula for the domain of~$\tau$.
\item Define $s\forces_\tau\alpha(\bar v)$,
 where $\alpha(\bar v)$ is a simple atomic $\lang(T')$~formula,
  according to the interpretation~$\tau$.
\end{itemize}
\end{eg}

\subsection{Forcing interpretations and polynomial simulation}

We now turn to the question of what is required of a forcing interpretation
if it is to imply a polynomial simulation between theories.

We begin by verifying that proofs of some basic facts, including generalizations
of \ref{item:fint/monotone}, \ref{item:fint/stermsubst}, and~\ref{item:fint/density}
to all formulas, can be found in polynomial time for any forcing interpretation.
Our convention that the language of a theory is always finite is implicitly used in the proofs of the lemmas below.

\begin{lem}\label{lem:LEM}
Let $\tau$ be a forcing interpretation of $\lang(T')$ in a theory~$T$.
Then there is a polynomial-time procedure which,
given an $\lang(T')$~formula $\theta(\bar v)$,
outputs a proof in $T$ of
   \begin{equation*}
    \fain s{\Cond[\tau]}{\fa{\bar v}{\neg\bigl(
     \dquot{s\forces_\tau\theta(\bar v)}
     \wedge\dquot{s\forces_\tau\neg\theta(\bar v)}
    \bigr).}}
   \end{equation*}
\end{lem}

\begin{proof}
Apply \ref{item:ftr/neg}.
\end{proof}

\begin{lem}\label{lem:monodense}
Let $\tau$ be a forcing interpretation of $\lang(T')$ in a theory~$T$.
Then there is a polynomial-time procedure which, given an $\lang(T')$~formula $\theta(\bar v)$,
outputs proofs in $T$ of:
\begin{enumerate}
\item\label{part:monodense/mono}
 \begin{math}
  \fain{s,s'}{\Cond[\tau]}{\fa{\bar v}{\bigl(
   s'\extendseq_\tau s\wedge\dquot{s\forces_\tau\theta(\bar v)}
   \then\dquot{s'\forces_\tau \theta(\bar v)}
  \bigr)}}
 \end{math}; and
\item\label{part:monodense/dense}
 \begin{math}
  \fain s{\Cond[\tau]}{\fa{\bar v}{\bigl(
   \faexteq[\tau]{s'}s{\exexteq[\tau]{s''}{s'}{\dquot{s''\forces_\tau \theta(\bar v)}}}
   \then \dquot{s\forces_\tau\theta(\bar v)}
  \bigr).}}
 \end{math}
\end{enumerate}
\end{lem}

\begin{proof}
These are part of Lemma~4.3 and Lemma~4.6 in Avigad~\cite{art:formforce},
except that we additionally need to pay attention to the computational complexity
of the proof constructions. The proofs are built using induction on the structure of $\theta$,
with the inductive step split into cases depending on the outermost logical connective in $\theta$.
In each case, the construction is straightforward.
\end{proof}

\begin{lem}\label{lem:stermsubst}
Let $\tau$ be a forcing interpretation of $\lang(T')$ in a theory~$T$.
There is a polynomial-time procedure which ---
given an $\lang(T')$~term~$t(\bar v)$ with exactly the free variables shown and
a simple $\lang(T')$~formula $\theta(\bar u,w_0)$
such that $\theta(\bar u,t(\bar v))$ is also simple, and both~$w$ and~$t(\bar v)$
    are substitutable for~$w_0$ in $\theta(\bar u,w_0)$ ---
outputs a proof in $T$ of
  \[
   \xfain s{\Cond[\tau]} \xfa{\bar u,\bar v,w}
   \bigl(
    \dquot{s\forces_\tau w=t(\bar v)}
    \then(
     \dquot{s\forces_\tau\theta(\bar u,w)}
     \nsc\dquot{s\forces_\tau\theta(\bar u,t(\bar v))}
    )
   \bigr).
  \]
\end{lem}

\begin{proof}
This is part of Lemma~4.3 in Avigad~\cite{art:formforce},
again with attention paid to the computational complexity of the construction.
As previously, one proceeds by induction on the structure of~$\theta$.
\end{proof}

\begin{defn}
A forcing interpretation~$\tau$ of a theory~$T'$ in a theory~$T$
is \defm{polynomial} if there is polynomial-time procedure which,
given any $\sigma\in T'$, outputs a proof in $T$ of the sentence
$\fain s{\Cond[\tau]}{s\forces_\tau \sigma}$.
\end{defn}


Note that if $T'$ is finitely axiomatized, then any forcing
interpretation of $T'$ is automatically a polynomial
forcing interpretation.


\begin{prop}\label{prop:logic+}
Let $\tau$ be a forcing interpretation of a theory~$T'$ in a theory~$T$.
For all $\lang(T')$~formulas $\phi(\bar v),\psi(\bar v)$,
 if $T'+\phi(\bar v)\proves\psi(\bar v)$,
  then $T$~proves
  \begin{equation}\label{eqn:proving-forcing-fla}
   \fain s{\Cond[\tau]}{\fa{\bar v}{\bigl(
    s\forces_\tau\phi(\bar v)\then s\forces_\tau\psi(\bar v)
   \bigr)}}.
  \end{equation}
If $\tau$ is polynomial,
 then a proof of \eqref{eqn:proving-forcing-fla} in $T'$ can be found
  in polynomial time given $\phi, \psi$, and a proof $\pi$
   of~$\psi(\bar v)$ from~$T'+\phi(\bar v)$.
\end{prop}

\begin{proof}
Let~$\pi^*$ be the proof obtained from $\pi$ according to Lemma \ref{lem:pf-rel}.
By \ref{item:ftr/nonsimple} and \ref{item:fint/theory},
for each non-logical axiom $\sigma$ used in $\pi^*$
we can find a proof in $T$ of $\fain s{\Cond} {\fa{\bar v} (s \forces \phi(\bar v) \then s \forces \sigma)}$.
If $\tau$ is a polynomial forcing interpretation of $T'$ in $T$,
these proofs can be found in polynomial time.

We then construct analogous proofs for each line $\sigma$ in $\pi^*$,
by a routine induction on the structure of~$\pi^*$;
cf.~Proposition~4.8 in Avigad~\cite{art:formforce}.
By \ref{item:ftr/nonsimple}, in the case of the last line this is a proof
of $\fain s{\Cond} {\fa{\bar v}(s \forces \phi(\bar v) \then s \forces \psi(\bar v))}$.
\end{proof}

We note the following important special case of Proposition \ref{prop:logic+}.

\begin{cor}\label{cor:logic}
Let $\tau$ be a forcing interpretation of~$T'$ in~$T$.
Then for each $\lang(T')$~sentence~$\sigma$, if $T'$ proves $\sigma$, then $T$ proves
$\fain s{\Cond[\tau]}{(s\forces_\tau\sigma)}$. Moreover, if $\tau$ is polynomial,
then a proof of $\fain s{\Cond[\tau]}{(s\forces_\tau\sigma)}$ in $T$
can be found algorithmically in polynomial time given a proof of $\sigma$ in $T'$.
\end{cor}


\begin{defn}\label{def:reflects}
Let $T,T'$ be theories and
    $\Gamma$ be a set of sentences in the language~$\lang(T)\cap\lang(T')$.
A forcing interpretation~$\tau$ of~$T'$ in~$T$
 is said to be \defm{$\Gamma$-reflecting}
  if, for all $\gamma\in\Gamma$, $T$ proves
  \begin{equation}\label{eqn:reflection}
   \fain s{\Cond[\tau]}{(s\forces_\tau\gamma)}\then\gamma.
  \end{equation}
The interpretation is \defm{polynomially $\Gamma$-reflecting}
if a $T$-proof of \eqref{eqn:reflection} can be found in polynomial
time on input $\gamma \in \Gamma$.
\end{defn}

\begin{thm}[essentially Avigad~{\cite[Section~10]{art:formforce}}]\label{thm:FI>nspeedup}
Let $T,T'$ be theories and
$\Gamma$~be a set of sentences in the language $\lang(T)\cap\lang(T')$.
If there is a $\Gamma$-reflecting forcing interpretation~$\tau$ of~$T'$ in~$T$,
then $T'$~is $\Gamma$-conservative over~$T$.
Moreover, if such~$\tau$ is polynomial and polynomially $\Gamma$-reflecting,
 then $T$~polynomially simulates~$T'$ with respect to~$\Gamma$.
\end{thm}

\begin{proof}
Combine \ref{item:fint/ex_Cond}, Corollary \ref{cor:logic}, and Definition \ref{def:reflects}.
\end{proof}


\section{Ramsey for pairs: polynomial simulation in $\RCA_0$}\label{sec:ramsey}

In this section, we prove the main theorem of our paper.

\begin{thm}\label{thm:simulation-rt22-rca0}
$\RCA_0$ polynomially simulates $\WKL_0 + \RT^2_2$ with respect to $\forall\Sigma^0_2$ sentences.
\end{thm}

This is a strengthening of the main result of \cite{PY} that $\WKL_0 + \RT^2_2$ is $\forall\Sigma^0_2$-conservative over
$\ind\Sigma^0_1$. Since $\RCA_0$ is polynomially simulated by $\ind\Sigma_1$ w.r.t.~$\Lone$ sentences
(see \cite{phd:Ignjatovic,ci:instrumentalism}
or consider the obvious interpretation of $\RCA_0$ in $\ind\Sigma_1$ and refer to~Section \ref{sec:forcing}),
Theorem \ref{thm:simulation-rt22-rca0} has the following immediate corollary.

\begin{cor}\label{cor:simulation-rt22-is1}
$\ind\Sigma_1$ polynomially simulates $\WKL_0 + \RT^2_2$ with respect to $\Pi_3$ sentences.
\end{cor}

Our proof of Theorem \ref{thm:simulation-rt22-rca0} requires a stronger and more explicit version
of a result in finite combinatorics used to prove the conservativity theorem of \cite{PY}. The combinatorial statement is formulated
using the concept of $\alpha$-largeness (originally introduced in \cite{KS81} as a tool to study the unprovability of the Paris--Harrington theorem),
which provides a framework for measuring the size
of finite subsets of $\IN$ by means of countable ordinals rather than just natural numbers.
We now recall some basic notions related to that framework.

For a given $\alpha<\omega^{\omega}$ and $m\in \IN$, define
$0[m]=0$, $\alpha[m]=\beta$ if $\alpha=\beta+1$,
and $\alpha[m]=\beta+\omega^{n-1}\cdot m$ if $\alpha=\beta+\omega^{n}$ for some $n\ge 1$.

\begin{defn}[$\ind\Delta_0+\exp$]
Let $\alpha<\omega^{\omega}$. A set $X = \{ x_0 < \dots < x_{\ell-1} \}\subseteq_{\fin}\IN$
is said to be \emph{$\alpha$-large} if $\alpha[x_{0}]\dots[x_{\ell-1}]=0$.
In other words, any finite set is $0$-large, and $X$ is $\alpha$-large when
\begin{itemize}
 \item $X\setminus \{\min X\}$ is $\beta$-large if $\alpha=\beta+1$,
 \item $X\setminus \{\min X\}$ is $(\beta+\omega^{n-1}\cdot\min X)$-large if $\alpha=\beta+\omega^{n}$ for $n \ge 1$.
\end{itemize}
\end{defn}

The proposition below lists some well-known and simple but important properties of $\alpha$-largeness. Their provability in $\ind\Delta_0 + \exp$ was pointed out e.g.~in \cite{KY:ramsey-combinatorics}.

\begin{prop}\label{prop:largeness} The following are provable in $\ind\Delta_0 + \exp$:
\begin{enumerate}[(i)]
\item\label{prop:largeness-monotone} If a set $X$ is $\alpha$-large and $X \subseteq Y \subseteq_\fin \IN$, then $Y$ is $\alpha$-large.
\item\label{prop:largeness-sums} If $\alpha = \omega^{n_k} + \dots + \omega^{n_0}$ where $n_k\ge \dots \ge n_0$, then a set $X$ is $\alpha$-large if and only if there exist sets $X_0,\dots,X_k$ such that $X = X_0 \sqcup \dots \sqcup X_k$, each $X_i$ is $\omega^{n_i}$-large, and for each $i < k$, $\max X_i < \min X_{i+1}$.

As a result, if $X$ is $\omega^n$-large and $k = \min X$, then $X\setminus\{k\} = X_0 \sqcup \dots \sqcup X_{k-1}$
where each $X_i$ is $\omega^{n-1}$-large,
and for each $i < k-1$, $\max X_i < \min X_{i+1}$.
\item\label{prop:largeness-split} If $X$ is $\omega^n \cdot 2$-large and $X = X_0 \cup X_1$, then at least one of $X_0, X_1$ is $\omega^n$-large.
\end{enumerate}
\end{prop}

The following result is crucial in the proof of the main theorem of \cite{PY}.

\begin{prop}\cite[Proposition 7.7]{PY}\label{prop:PY-largeness}
For every natural number $n$ there exists a natural number $m$ such that $\ind\Sigma_1$ proves: for every $X\subseteq_{\fin}\IN$ with $\min X \ge 3$, if $X$ is $\omega^m$-large, then every colouring $f \colon [X]^2 \to 2$ has an $\omega^n$-large homogeneous set.
\end{prop}

As the analysis in the remainder of this section will reveal, proving Theorem \ref{thm:simulation-rt22-rca0} involves obtaining an upper bound on $m$ in terms of $n$. The statement of Proposition \ref{prop:PY-largeness} gives no such upper bound, and any bound that could be derived from a simple-minded analysis of the proof in \cite{PY} would be very weak --- certainly not even elementary recursive. On the other hand, proving Theorem~\ref{thm:simulation-rt22-rca0} requires a bound significantly better than exponential. We will use the bound obtained in \cite{KY:ramsey-combinatorics}:

\begin{thm}\cite[Theorem 1.6]{KY:ramsey-combinatorics}\label{thm:RT22-largeness-main}
Provably in $\ind\Sigma_1$, if $X\subseteq_{\fin}\IN$ is $\omega^{300x}$-large, then every colouring $f \colon [X]^2 \to 2$ has an $\omega^x$-large homogeneous set.
\end{thm}

\begin{remark}
Note that in \cite{KY:ramsey-combinatorics}, the theorem includes the additional assumption that $\min X \ge 3$. However, using the proof presented in \cite{KY:ramsey-combinatorics} and some very straightforward combinatorics of $\alpha$-large sets, it is easy to verify that the assumption is in fact unnecessary. To avoid annoying technicalities, we nevertheless assume that any set $X$ mentioned in the context of $\alpha$-largeness satisfies
$\min X \ge 1$.
\end{remark}

Our proof of Theorem \ref{thm:simulation-rt22-rca0} will involve an argument by cases,
in which the case distinction is based on how much induction is available.
We introduce an auxiliary theory, originally motivated by the case in which a certain amount of induction \emph{fails},
which turns out to be useful in the other case as well. The theories and polynomial simulation relationships
we consider in the proof are summarized in Figure 1.

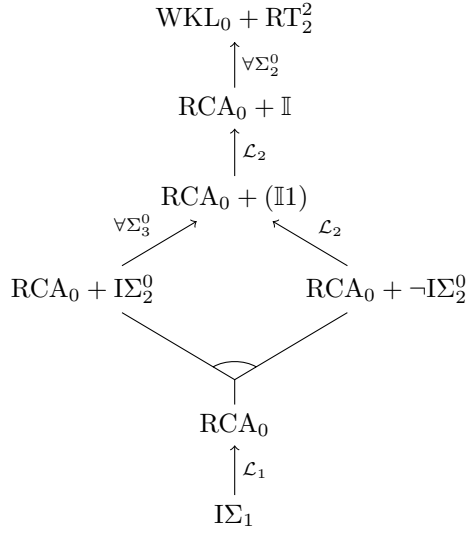
\begin{figure}[tph]\centering
\begin{tikzpicture}[xscale=2,yscale=1.2]
\node(IS1)    at ( 0,0)   {$\ind\Sigma_1$};
\node(RCA0)   at ( 0,1)   {$\RCA_0$};
\node[minimum width=8em](RCA0+)
              at (-1,2.5) {$\RCA_0+\ind\Sigma^0_2$};
\node[minimum width=8em](RCA0-)
              at ( 1,2.5) {$\RCA_0+\neg\ind\Sigma^0_2$};
\node(RCA0I1) at ( 0,3.5) {$\RCA_0+(\I1)$};
\node(RCA0I)  at ( 0,4.5) {$\RCA_0+\I$};
\node(RT22)   at ( 0,5.5) {$\WKL_0+\RT^2_2$};
\begin{scope}[->,node font=\scriptsize]
\draw(IS1)   --node[right]{$\Lone$}                  (RCA0);
\draw(RCA0+) --node[above left]{$\forall\Sigma^0_3$} (RCA0I1);
\draw(RCA0-) --node[above right]{$\Ltwo$}            (RCA0I1);
\draw(RCA0I1)--node[right]{$\Ltwo$}                  (RCA0I);
\draw(RCA0I) --node[right]{$\forall\Sigma^0_2$}      (RT22);
\end{scope}
\draw(RCA0)--++(0,.5) coordinate(pivot)
     (pivot)--(RCA0+)
     (pivot)--(RCA0-)
     (pivot)+(45:.2) arc [radius=.2, start angle=45, end angle=135];
\end{tikzpicture}
\caption{Polynomial simulations between the various theories in Section~\ref{sec:ramsey}}
\end{figure}

\begin{defn}\label{defn:cut}
$\RCA_0 + \I$ is a theory in the language of second-order arithmetic extended by
a new unary predicate $\I$ over the first-order sort.

The axioms of $\RCA_0 + \I$ are those of $\RCA_0 $ plus the following statements:
\begin{enumerate}[($\I$1)]
  \item $\I$ is a nonempty proper cut in the first-order universe,
  \item $\I$ is closed under addition,
  \item for every infinite set $S$, there exists finite $X \subseteq S$ which is $\omega^x$-large for some $x > \I$.
\end{enumerate}
\end{defn}

For our purposes, the most important property of the theory $\RCA_0 + \I$ will be that it
polynomially simulates $\WKL_0 + \RT^2_2$ with respect to proofs of $\forall\Sigma^0_2$ sentences.
Before proving that fact, which is Lemma \ref{lem:simulation-rt22-I} below,
we verify some more basic properties of $\RCA_0 + \I$.

\begin{lem}\label{lem:simulation-I-I1}
$\RCA_0 + (\I1)$ polynomially simulates $\RCA_0 + \I$ with respect to $\Ltwo$ sentences.
\end{lem}
\begin{proof}
It is enough to show that there is an interpretation of $\RCA_0 + \I$ in $\RCA_0 + (\I1)$
that is the identity interpretation with respect to all symbols of $\Ltwo$.

So, let $\mathbb{J}$ consist of those numbers $x$ for which every infinite set contains an $\omega^x$-large finite subset;
this is easily seen to be a definable cut in $\RCA_0$.

Thus, $\I \cap \mathbb{J}$ is, provably in $\RCA_0 + (\I1)$, a proper definable cut
with the property that every infinite set contains an $\omega^x$-large subset for each $x \in \I \cap \mathbb{J}$.
It follows that every infinite set $X$ actually contains an $\omega^x$-large subset for some $x > \I \cap \mathbb{J}$,
because the formula ``$X$ contains an $\omega^x$-large subset'' is $\Sigma^0_1$,
so by $\ind\Sigma^0_1$ it cannot define exactly the cut $\mathbb{I} \cap \mathbb{J}$.
Let $\mathbb{K}$ be a cut closed under addition
obtained from $\I \cap \mathbb{J}$ by applying Solovay's technique of shortening cuts:
in other words, put $\mathbb{K}=\{a\in\mathbb{I}\cap\mathbb{J}: \forall x\in\mathbb{I}\cap\mathbb{J}(a+x\in \mathbb{I}\cap\mathbb{J})\}$.
Then $\RCA_0 + (\I1)$ proves the axioms of $\RCA_0 + \I$ with $\mathbb{K}$ substituted for $\I$.
\end{proof}

\begin{lem}\label{lem:case-distinction}
$\RCA_0 + (\I1)$ is polynomially simulated by:
\begin{enumerate}[(a)]
\item $\RCA_0 + \ind\Sigma^0_2$ with respect to $\forall\Sigma^0_3$ sentences,\label{part:case-distinction/+}
\item $\RCA_0 + \neg\ind\Sigma^0_2$ with respect to $\Ltwo$ sentences.\label{part:case-distinction/-}
\end{enumerate}
\end{lem}

\begin{proof}
The proof of~\ref{part:case-distinction/-} is essentially immediate: by interpreting $\I$ as the intersection of all $\Sigma^0_2$-definable
cuts, which is a definable proper cut in $\RCA_0 + \neg\ind\Sigma^0_2$, we obtain an
interpretation of $\RCA_0 + (\I1)$ in $\RCA_0 + \neg\ind\Sigma^0_2$
that is the identity with respect to $\Ltwo$.

In the case 
of~\ref{part:case-distinction/+},
let us first explain the underlying model-theoretic argument,
which shows how to construct a model of $\RCA_0+(\I1)$
from a model of $\RCA_0+\ind\Sigma^0_2$
while preserving the truth of a fixed $\exists\Pi^0_3$~sentence.
We then give some hints as to how this is reflected on the syntactic level.

Let $(M,\X)\models\RCA_0+\ind\Sigma^0_2+\ex X{\delta(X)}$,
where $\delta\in\Sigma^0_4$. Take $A\in\X$ such that $(M,\X)\models\delta(A)$.
From Beklemishev~\cite[Theorem~5.1]{art:pfthy-coll}, we know that
$\ind\Sigma_2$ proves the uniform $\Pi_4$~reflection principle for~$\ind\Sigma_1$.
A relativization of this implies
\begin{math}
(M,\X)\models\Con(\ind\Sigma^0_1+\delta(P))
\end{math},
where $P$ is a new predicate.
An application of the Low Arithmetized Completeness Theorem~\cite[Theorem~I.4.27]{book:hajek+pudlak}
then gives an end-extension $M'\eextneq M$ and $A'\subseteq M'$
such that $(M',A')\models\ind\Sigma^0_1+\delta(P)$.
Thus $(M',\Delta_1(P)\hyp\Def(M',A'),M) \models \RCA_0+(\I1)+\ex X{\delta(X)}$.
Moreover, the universe and operations in $(M',A')$ are definable in $M$,
as is the range of a bijection between $M$ and an initial segment of $M'$
closed under $+^{M'}$.

Crucially, the formulas used to define $(M',A')$ are not only independent
of the model $M$ of $\Con(\ind\Sigma^0_1+\delta(P))$ that we consider,
but they are also polynomial-time uniform in $\delta$, in the sense that they can be constructed
in polynomial time given $\delta$ as input. (Each of the formulas arises essentially
by substituting the G\"odel number of $\ex X{\delta(X)}$ into a $\Sigma_2$ formula $\gamma(x)$
describing a low branch of a binary tree that is $\Delta_1$-definable with first-order parameter $x$
and is infinite as long as $x$ is the G\"odel number of a sentence consistent with $\ind\Sigma^0_1$.)
So, if $\pi$ is a proof of $\bot$ from $\RCA_0+(\I1)+\ex X{\delta(X)}$,
we can build a proof of $\bot$ from $\ind\Sigma_2(P)+{\delta(P)}$
by taking $\pi^*$ and ``relativizing it to $(M',\Delta_1(P)\hyp\Def(M',P),M)$''.
That is, we relativize first-order quantifiers to the formula defining the universe of $M'$
and replace first-order atoms by the formulas defining operations in $M'$;
replace second-order quantifiers by quantification over pairs of a $\Sigma_1(P)$ formula
and an equivalent $\Pi_1(P)$ formula, changing atoms $x \in X$
(for $X$ different from $P$) into appropriate instances of the $\Sigma_1(P)$-universal formula;
replace $\I$ by the formula defining the range of the inclusion $M \hookrightarrow M'$.
Additionally, it is necessary to add polynomially many new proof lines in order
to derive the translated axioms of $\RCA_0+(\I1)+\ex X{\delta(X)}$ from $\ind\Sigma_2(P)+{\delta(P)}$,
and to deriving translations of conclusions from translations of premises
for each inference in $\pi^*$. The details of the latter task are quite standard
and somewhat similar to those in the proof that $\ind\Sigma_1$ polynomially
simulates $\RCA_0$ based on the usual interpretation of $\RCA_0$
in $\ind\Sigma_1$, as described in e.g.~\cite[Chapter 3]{phd:Ignjatovic}.
\end{proof}




\begin{lem}\label{lem:simulation-rt22-I}
$\RCA_0 + \I$ polynomially simulates $\WKL_0 + \RT^2_2$ with respect to $\forall\Sigma^0_2$ sentences.
\end{lem}

To prove Lemma \ref{lem:simulation-rt22-I}, we now define a particular forcing notion in $\RCA_0 + \I$. We will then use a series of lemmas to show that this provides a polynomial forcing interpretation of $\WKL_0 + \RT^2_2$ in the sense of Section \ref{sec:forcing} and, moreover, that the forcing interpretation is polynomially $\forall \Sigma^0_2$-reflecting in the sense of Definition~\ref{def:reflects}.

\begin{defn}\label{def:our-forcing}
We let a \emph{finite} set $s$  be a forcing condition ($s\in\Cond$) if and only if $s$ is $\omega^x$-large for some $x > \I$. The relation
$\extendseq$ is defined simply as $\subseteq$. Note that there is no $\extendseq$-largest condition.

There are two sorts of names, a first- and a second-order sort. A name of the first-order sort is simply a natural number.
A name of the second-sort is also simply a natural number, this time viewed as coding a finite set according to the
Ackermann interpretation. (Formally, one may think of the first-order names being $\tuple{0,v}$ and second-order names being $\tuple{1,v}$
for various numbers $v$.) To avoid notational confusion, we will write $v$ for the first-order and $V$ for the second-order names.

We say that $s \forces v \defd$ if $s \cap [1,v]$ is not a condition, that is, if it is not $\omega^x$-large for any $x > \I$.
On the other hand, $s \forces V \defd$ holds always.

For each simple atomic formula  $\varphi(\bar v)$,
we define $s \forces \varphi$ to be $(s \forces \bar v \defd) \land \varphi(\bar v)$.
(In the special case of the simple atomic formula $v \in V$, this would be more precisely stated as
$(s\forces v \defd \land s \forces V \defd) \land v \in_{\mathrm{Ack}} V$,
where $\in_{\mathrm{Ack}}$ is the standard Ackermann interpretation
of $\in$ in arithmetic. We shall ignore this detail from now on.)
\end{defn}

Before showing that this forcing notion provides a polynomial forcing interpretation of $\WKL_{0}+\RT^{2}_{2}$, we first discuss the model-theoretic intuition behind it.
Let $(M,\X,I)$ be a countable model of $\RCA_0 + \I$.
The set of forcing conditions $\Cond$ and the relation $\extendseq$ of Definition~\ref{def:our-forcing}
are both definable in $(M,\X,I)$.
Take an $(M,\X,I)$-generic filter $G$ of $(\Cond,\extendseq)$, and put $I_{G}=\sup_{M}\{\min s: s\in G\}$.
One may then check that $(I_{G},\Cod(M/I_{G}))\models \WKL_{0}+\RT^{2}_{2}$
by simulating the inductive construction of a cut in the proof of \cite[Theorem~3.3]{KY:ramsey-combinatorics},
where various choices made in the inductive steps are now replaced by the genericity of $G$.
One may also check that for any $v\in M$, $v\in I_{G}$
if and only if $s \cap [1,v] \notin \Cond$ for some $s \in G$,
and that $s\cap I_{G}$ is cofinal in $I_{G}$ for any $s\in G$.

Thus, in Definition~\ref{def:our-forcing}, $s \forces v \defd$ is intended to hold only if $s \in G$ guarantees that $v$ will be in $I_G$. On the other hand, the name $V$ is intended to refer to $V \cap I_G$, which will always exist. The structure of $+, \cdot, \le$ on $I_G$ is inherited from the ground model, which motivates the trivial definition of forcing atomic formulas.

\begin{remark}
By Proposition \ref{prop:largeness}\ref{prop:largeness-split}, whenever an $\omega^x$-large set is split into two subsets, one of them is $\omega^{x-1}$-large. As a consequence, if $s \forces v \defd$, then $s \setminus [1,v]$ is a condition.
\end{remark}

\begin{lem}\label{lem:rca-I-fint-trivial}
The set of conditions $\Cond$ and the relations $\extendseq$, $\forces$ of Definition \ref{def:our-forcing} determine a forcing interpretation of $\Ltwo$ in the theory $\RCA_0 + \I$.
\end{lem}

\begin{proof}
Checking that $\Cond$, $\extendseq$, $\forces$ determine a forcing translation from $\Ltwo$ into $\Ltwo \cup \{\I\}$ is unproblematic, clause \ref{item:fint/theory} of the definition of forcing interpretation trivializes, and most of the other clauses are also very easy to verify in $\RCA_0 + \I$. We discuss \ref{item:fint/functions} and \ref{item:fint/densedefd}, which are perhaps less obvious than the rest.

To verify \ref{item:fint/functions}, note that if the set $(x,y]$ is $\omega$-large, then $y  \ge 2x$, and if it is $\omega^2$-large, then $y > x2^x \ge x^2$. Assume that $s$ is $\omega^x$-large for some $x > \I$ and that $s \forces v\defd, s \forces w \defd$ with say $v \le w$. Since $s \cap [1,w]$ is not a condition and neither is $s \cap (w,w^2]$ (being a subset of the set $(w,w^2]$, which is not even $\omega^2$-large), it follows from Proposition \ref{prop:largeness}\ref{prop:largeness-split} that $s \cap [1,w^2]$ is not a condition. Thus, $s \forces (w^2)\defd$ and a fortiori $s$ forces both $(v+w)\defd$ and $(vw)\defd$.
This proves the ``existence'' part of \ref{item:fint/functions}. The ``uniqueness'' part follows easily from \ref{item:ftr/forall}, \ref{item:ftr/then}, and the definition of $\forces$.

To show \ref{item:fint/densedefd}, assume $s \nforces v \defd$. Then $s' \defeq s \cap [1,v]$ is a condition with $s' \extendseq s$. However, for any condition $s'' \extendseq s'$, we have $s'' \cap [1,v] = s''$, so certainly $s'' \nforces v \defd$.
\end{proof}

We can use the proof of Lemma \ref{lem:rca-I-fint-trivial} to show that, actually, a slightly stronger version of  \ref{item:fint/functions} holds.

\begin{lem}\label{lem:terms-defined}
There is a polynomial-time procedure which, given a term $t(\bar v)$, outputs an $\RCA_0 + \I$ proof of the statement that, for any condition $s$,
if $s \forces \bar v \defd$, then $s \forces t \defd$.
\end{lem}
\begin{proof}
Using a standard argument, we can prove in $\RCA_0$ that $t < (\max \bar v)^n$, where $n$ is the number of symbols in $t$. By the proof of Lemma \ref{lem:rca-I-fint-trivial}, we can also prove in $\RCA_0 + \I$ that for any $s$ and $w$, if $s \forces w \defd$, then $s \forces (w^2) \defd$. Iterating this reasoning $\log n$ times, we prove that if $s \forces w \defd$, then $s \forces (w^n) \defd$. We apply this to $w \defeq \max \bar v$ in order to deduce $s \forces (\max \bar v)^n \defd$ and hence $s \forces t \defd$.
\end{proof}

\begin{lem}\label{lem:D0elem}
There is a polynomial-time procedure which, given a $\Delta_0$ formula $\theta(\bar v, \bar V)$, outputs an $\RCA_0 + \I$ proof
of \[\forall s\, \forall \bar v \, \forall \bar V \, [s \forces \bar v \defd \then (\theta(\bar v, \bar V) \leftrightarrow (s \forces \theta(\bar v, \bar V)))].\]
\end{lem}

\begin{proof}
The construction of the proof proceeds by induction on the structure of $\theta$. As the reader will be able to check, it can be carried out in polynomial-time. We  give an informal description of how the proof is built for the atomic step and for each case of the inductive step.

\paragraph{Suppose $\theta$ is atomic.}
If $\theta$ is a simple formula,
then the equivalence between $\theta$ and $s \! \forces \! \theta$ follows immediately from the definition
of our forcing relation. Otherwise, if $\theta$ is say $t_1(\bar v) \le t_2(\bar v)$ (the other cases are similar or easier),
then $\theta^*$ is defined to be
\[\forall x_1\, \forall x_2\,(t_1[x_1] \then (t_2[x_2] \then  x_1 \le x_2)),\]
(see Definition \ref{defn:simple-tr})
and by \ref{item:ftr/nonsimple} we have to prove that a condition forces $\theta^*$ exactly if it forces
$\bar v \defd$ and $\theta$ in fact holds. To achieve this, one first shows by induction on subterms
$t'(\bar u)$ of $t_1$ and $t_2$ (where $\bar u$ is a subtuple of $\bar v$) that,
for the appropriate variable $x'$ and any $s \in \Cond$,
\begin{equation}\label{eqn:forcing-atomic}
s \forces t'[x'] \quad\textrm{ iff }\quad (s \forces \bar u \defd \land x' = t(\bar u)).
\end{equation}
The inductive step is proved using \ref{item:ftr/forall}, \ref{item:ftr/then},
the definition of $\forces$, and Lemma~\ref{lem:terms-defined}.
Once \eqref{eqn:forcing-atomic} is proved for $t'$ equal to $t_1$ and $t_2$,
one can prove the equivalence of $s \forces \theta^*$ with $(s \forces \bar v \defd \land t_1 \le t_2)$
by one more series of appeals to \ref{item:ftr/forall}, \ref{item:ftr/then},
the definition of $\forces$, and Lemma \ref{lem:terms-defined}.

\paragraph{Suppose $\theta$ is $\neg\eta$.} The inductive assumption gives us a proof that whenever $s \forces \bar v \defd$, then
$s \forces \eta(\bar v, \bar V)$ is equivalent to $\eta(\bar v, \bar V)$. This equivalence generalizes to all $s' \extendseq s$, since for each such $s'$ we have
$s' \forces \bar v \defd$ as well.

So, let $s \forces \bar v \defd$. Assuming $\neg\eta$, we have $s' \nforces \eta$ for any $s'\extendseq s$, and thus $s \forces \neg \eta$. On the other hand, if $s \forces \neg \eta$, then by the definition of forcing $s \nforces \eta$ and hence $\neg \eta$.

\paragraph{Suppose $\theta$ is $\eta(\bar v,\bar V)\then\xi(\bar v,\bar V)$.} The inductive assumption gives us a proof that
whenever $s \forces \bar v \defd$ and $s' \extendseq s$, then $s \forces \eta(\bar v, \bar V)$ is equivalent to $\eta(\bar v, \bar V)$
and likewise for $\xi$.

Let $s \forces \bar v \defd$. Assume $\eta \then \xi$. If $s'\extendseq s$ forces $\eta$, then $\eta$ holds and by our assumption so does $\xi$, which implies $s' \forces \xi$. This shows $s \forces \eta \then \xi$. Conversely, assume $\eta \land \neg \xi$. Then $s \forces \eta$ but for each $s' \extendseq s$ we have $s' \nforces \xi$. So, $s \nforces \eta \then \xi$.

\paragraph{Suppose $\theta(\bar v,\bar V)$
is $\fale w{t}{\eta(w,\bar v, \bar V)}$.}
Recall that $\fale w{t}\ldots$ is shorthand for $\forall w \,(w \le t \then \ldots)$.
(Here $t=t(\bar v)$ is a term in the variables $\bar v$.)
The atomic step and the inductive assumption, respectively,
give us proofs that if $s \forces \bar v \defd$ and $s \forces w \defd$, then
\begin{align*}
s \forces w \le t & \textrm{ is equivalent to } w \le t,  \\
s \forces \eta(w,\ldots) & \textrm{ is equivalent to } \eta(w,\ldots).
\end{align*}
As in the previous cases, the equivalences generalize to all $s' \extendseq s$.

Moreover, by Lemma \ref{lem:terms-defined},
we have a polynomial-time constructible proof that if $s \forces \bar v \defd$,
then $s' \extendseq s$ implies
$s' \forces t(\bar v) \defd$ and thus $s' \forces w \defd$ for any $w \le t(\bar v)$.

Let $s \forces \bar v \defd$ and assume $\fale w{t}{\eta(w,\ldots)}$.
Take $s'\extendseq s$ and $w$ such that
$s'\forces w\defd$ and $s'\forces w \le t$.
Then the atomic step gives $w \le t$,
which, by our assumption, implies $\eta(w,\ldots)$.
Thus $s' \forces \eta(w,\ldots)$,
which is what we want.

Conversely, assume $w \le t$ is such that $\neg\eta(w,\ldots)$.
Then $s \forces w \le t$.
Also, for each $s' \extendseq s$ we have $s' \forces w \defd$, so $s' \nforces \eta(w,\ldots)$.
Thus, $s \nforces \fale w{t}{\eta(w,\ldots)}$.
\end{proof}

\begin{lem}\label{lem:rca-I-fint-rt22}
The relations $\Cond$, $\extendseq$, $\forces$ of Definition \ref{def:our-forcing}
determine a polynomial forcing interpretation of $\WKL_0 + \RT^2_2$ in $\RCA_0 + \I$.
\end{lem}

\begin{proof}
We have already shown in Lemma \ref{lem:rca-I-fint-trivial}
that $\Cond$, $\extendseq$, $\forces$ determine a forcing interpretation of $\Ltwo$ in $\RCA_0 + \I$.
So if we want to argue that this is actually a forcing interpretation of $\WKL_0 + \RT^2_2$, we only have to check \ref{item:fint/theory}.
In other words, we must show that any condition $s$ forces each of the axioms of $\WKL_0 + \RT^2_2$.
This is immediate for the axiomatization of non-negative parts of discrete ordered rings, so it remains to deal with $\Delta^0_1$~comprehension,
$\Sigma^0_1$~induction, Weak K\"onig's Lemma, and $\RT^2_2$.

We reason in $\RCA_0 + \I$, sometimes making tacit use of the following simple observation: for any condition $s$ and any $k$, a condition $s' \subseteq s$ with $s' \forces k\defd$ exists exactly if $s \setminus [1,k]$ is a condition.

We first deal with the $\Sigma^0_1$~separation scheme, which implies both $\Delta^0_1$~com\-pre\-hen\-sion and Weak K\"onig's Lemma. Consider a $\Sigma^0_1$~separation axiom:
\begin{multline*}\forall \bar V \, \forall \bar v \, [\forall k\,  (\sigma_1(k,\bar v, \bar V) \then \neg \sigma_2(k,\bar v, \bar V)) \\
\then \exists W \, \forall k \, ((\sigma_1(k,\bar v, \bar V) \then k \in W) \land (k \in W \then \neg \sigma_2(k,\bar v, \bar V)))],\end{multline*}
where $\sigma_1 \defeq \exists \ell\, \eta_1(\ell, k, \ldots)$ and $\sigma_2 \defeq \exists \ell\, \eta_2(\ell, k, \ldots)$ are $\Sigma^0_1$.
Let $s$ be a condition and assume that $s$ forces $\bar V \defd, \bar v \defd$ and that
\begin{equation}\label{eqn:forcing-ca} s \forces \forall k\, (\sigma_1(k,\bar v, \bar V) \then \neg \sigma_2(k,\bar v, \bar V)).\end{equation}
By the definition of forcing and Lemma \ref{lem:D0elem}, (\ref{eqn:forcing-ca}) implies the following: for every $k,\ell_1 \le \max s$ such that $s \setminus [1, \max(k,\ell_1)]$ is a condition, if $\eta_1(\ell_1,k,\ldots)$ holds, then there is no $\ell_2 \le \max s$ such that  $\eta_2(\ell_2,k,\ldots)$ and $s \setminus [1,\ell_2]$ is a condition. In particular, it cannot be the case that $\ell_1 \ge \ell_2$. Thus, taking
\[W_0 \defeq \{k \le \max s: \exists \ell_1 \! \le \! \max s\, (\eta_1(\ell_1,k,\ldots) \land \forall \ell_2 \! < \! \ell_1\, \neg \eta_2(\ell_2,k,\ldots))\},\]we conclude that $s \forces \forall k \, ((\sigma_1(k,\ldots) \then k \in W_0) \land (k \in W_0 \then \neg \sigma_2(k,\ldots)))$.

We now turn to $\Sigma^0_1$~induction. Consider the induction axiom for the $\Sigma^0_1$ formula $\sigma \defeq \exists \ell\, \eta(\ell, k, \bar v, \bar V)$ with respect to the variable $k$. Let $s$ be a condition forcing $\bar v \defd, \bar V \defd$ and let $j$ be such that $s \setminus [1,j]$ is a condition.
Let $m = \min (s \setminus [1,j])$ and let $x > \I$ be such that $s \setminus [1,j]$ is $\omega^x$-large. It follows that $s \setminus [1,m]$
is $\omega^{x-1}m$-large, which by Proposition \ref{prop:largeness}\ref{prop:largeness-sums} means that we can write $s \setminus [1,m]$ as a disjoint union
\[s \setminus [1,m] = s_0 \sqcup \dots \sqcup s_{m-1},\]
where each set $s_i$ is $\omega^{x-1}$-large (in particular, it is a condition) and, for each $i < m-1$, $\max s_i < \min s_{i+1}$. Define:
\[W_0 \defeq \{\ell: \exists k \! < \! j \, (\eta(\ell,k,\ldots) \land \forall \ell' \! < \! \ell \,\neg \eta(\ell',k,\ldots))\}.\]
Since $W_0$ has at most $j$ elements and $j < m$, the finite pigeonhole principle implies
that $W_0 \cap [\min s_{i_0}, \max s_{i_0}] = \emptyset$ for some $i_0 < m$. Note that for every $k < j$ we have
\begin{equation}\label{eqn:forcing-ind}\exists \ell \! < \! \min s_{i_0}\, \eta(\ell, k, \ldots) \textrm{ iff } \exists \ell \! \le \! \max s_{i_0}\, \eta(\ell, k, \ldots).\end{equation}
By bounded induction on $k$, either $\neg \exists \ell \! \le \! \max s_{i_0}\, \eta(\ell, 0, \ldots)$ or there is a maximal $k < j$ such that $\exists \ell \! \le \! \max s_{i_0}\, \eta(\ell, k, \ldots)$. By the property of $s_{i_0}$ stated in (\ref{eqn:forcing-ind}), it follows that $s_{i_0}$ forces one of the three statements $\neg \sigma(0,\ldots)$,
$\sigma(j{-}1,\ldots)$,
or $\exists k \! < \! j{-}1\, (\sigma(k,\ldots) \land \neg\sigma(k{+}1,\ldots))$.
In each case, the existence of $s_{i_0}$ implies that $s$ forces induction for $\sigma$ below $j$.

Finally, we deal with $\RT^2_2$. Assume that $s$ is a condition such that $s \forces f\defd$ and $s \forces f\colon [\IN]^2 \to 2$. This means that
$f$ is defined on each pair $\tuple{i,j}$ such that $i<j<\max s$ and $s \setminus [1,j]$ is a condition. Writing $s = s_0 \sqcup s_1$ where $\max s_0 < \min s_1$ and $s_0,s_1$ are both conditions, we can conclude that $f$ is defined on all arguments below $\max s_0$.

Let $x > \I$ be such that $s_0$ is $\omega^x$-large. By Theorem \ref{thm:RT22-largeness-main}, there exists an $\omega^{x/300}$-large set
$s' \subseteq s_0$ which is homogeneous w.r.t~$f$. By $(\I2)$, $\I$ is closed under addition, so $x/300 > \I$,
and hence $s'$ is a forcing condition.

It is easily verified that $s'$ forces ``$s'$ is homogeneous w.r.t~$f$''. However, we also have:
\[s' \forces \forall k\, \exists \ell\, (\ell  >  k \land \ell \in s').\]
To see this, take a condition $s'' \subseteq s'$ and $k$ such that $s'' \forces k \defd$. Then $s'' \setminus [1,k]$ is a condition, and so is $s'' \setminus [1,\ell]$ for $\ell\defeq\min (s'' \setminus [1,k])$. However, $s'' \setminus [1,\ell] \forces \ell \defd$, and since $\ell > k \land \ell \in s'$ is a true $\Delta_0$ statement, Lemma \ref{lem:D0elem} implies that it is forced by $s'' \setminus [1,\ell]$ as well.

This completes the proof that $\Cond, \extendseq, \forces$ give a forcing interpretation of $\WKL_0 + \RT^2_2$ in $\RCA_0 + \I$.
Polynomiality of the interpretation is immediate if we assume that $\WKL_0 + \RT^2_2$ is finitely axiomatized.
We may make this assumption w.l.o.g.~because both $\WKL$ and $\RT^2_2$ are single axioms,
while $\RCA_0$ can be axiomatized by using a finite number of instances of $\Delta^0_1$~comprehension
and $\Sigma^0_1$~induction in such a way that a proof of any of the other instances can be constructed in polynomial time.
\end{proof}

\begin{lem}\label{lem:rca-I-poly-fint-rt22}
The forcing interpretation of $\WKL_0 + \RT^2_2$ in $\RCA_0 + \I$ given by $\Cond$, $\extendseq$, $\forces$ of Definition \ref{def:our-forcing} is polynomially $\forall\Sigma^0_2$-reflecting.
\end{lem}

\begin{proof}
Let $\varphi \defeq \exists W \, \exists w \, \forall y \, \exists z\, \theta(W,w,y,z)$ be an $\exists \Pi^0_2$ sentence.
We sketch a proof in $\RCA_0 + \I$ that, assuming $\varphi$,
there is a condition $s$ such that $s \nforces \forall W\, \forall w\, \exists y\, \forall z\, \neg \theta$.
On the basis of the sketch,  it will be routine to verify
that the proof can be constructed in polynomial time on input $\varphi$.

Let $A,k$ be such that $\forall y \, \exists z \, \theta(A,k,y,z)$. By $\Sigma^0_1$ collection, for each $\ell$ there exists $m$ such that $\fale y \ell {\exle z m  {\theta(A,k,y,z)}}$. Use primitive recursion to define a sequence of numbers by:
\begin{align*}
k_0 \defeq & \quad k, \\
k_{n+1} \defeq & \quad \textrm{smallest } m >k_n \textrm{ such that } \fale y {k_n} {\exle z m  {\theta(A,k,y,z)}}.
\end{align*}
The axioms of $\RCA_0$ imply the existence of the set $S = \{k_n: n \in \IN\}$, which is clearly infinite.
Use the axiom $(\I3)$ of $\RCA + \I$
to obtain a finite set $s \subseteq S$ which is a forcing condition.
Below, we abuse notation and write $A$ for the finite set $A{\upharpoonright}_{\max s}$.

We claim that $s \forces \forall y \, \exists z \, \theta(A,k,y,z)$,
which is enough to imply $s \nforces \forall W\, \forall w\, \exists y\, \forall z\, \neg \theta$.
Take $s' \extendseq s$ and $\ell$ such that $s' \forces \ell \defd$.
Since $s' \cap [1,\ell]$ is not a condition, $s' \setminus [1,\ell]$
must be one. Let $m_1 < m_2$ be the two smallest
elements of $s' \setminus [1,\ell]$.
By Proposition \ref{prop:largeness}\ref{prop:largeness-split}, $s' \setminus [1,m_2]$ is also a condition
and $(s' \setminus [1,m_2]) \forces m_2 \defd$.
It follows from the definition of $s$ that $\fale y {m_1} {\exle z {m_2}  {\theta(A,k,y,z)}}$.
In particular, there is some $m \le m_2$ such that
$(s' \setminus [1,m_2]) \forces m \defd$ and $\theta(A,k,\ell,m)$.
By Lemma \ref{lem:D0elem}, we get $(s' \setminus [1,m_2]) \forces \theta(A,k,\ell,m)$.
This implies $(s' \setminus [1,m_2]) \forces \exists z \, \theta(A,k,\ell,z)$,
which is what we wanted.
\end{proof}

\begin{proof}[Proof of Lemma \ref{lem:simulation-rt22-I}]
This follows directly from Lemma \ref{lem:rca-I-fint-rt22}, Lemma \ref{lem:rca-I-poly-fint-rt22},
and Theorem \ref{thm:FI>nspeedup}.
\end{proof}

\begin{proof}[Proof of Theorem \ref{thm:simulation-rt22-rca0}]
By Lemma \ref{lem:simulation-rt22-I} there is a polynomial-time procedure which,
given a proof $\pi$ of a $\forall\Sigma^0_2$ sentence $\psi$ in $\WKL_0 + \RT^2_2$,
outputs a proof $\pi'$ of $\psi$ in $\RCA_0 + \I$. By Lemma \ref{lem:case-distinction},
a further polynomial-time procedure outputs
a proof $\pi''$ of $\psi$ in $\RCA_0 + \ind\Sigma^0_2$ and
a proof $\pi'''$ of $\psi$ in $\RCA_0 + \neg\ind\Sigma^0_2$.
Combine $\pi'', \pi'''$, and a case distinction to obtain
a proof of $\psi$ in $\RCA_0$.
\end{proof}

\section{Ramsey for pairs: speedup over $\RCA^*_0$}

We will now prove that Theorem \ref{thm:simulation-rt22-rca0} breaks down completely in the absence
of $\Sigma^0_1$ induction, that is, if the base theory $\RCA_0$ is replaced by $\RCA^*_0$,
 although $\WKL_0^*+\RT^2_2$ is also $\forall\Sigma^0_2$-conservative
  over $\RCA_0^*$~\cite{Y-MLQ13},

\begin{thm}\label{thm:speedup}
$\RT^2_2$ has non-elementary speed-up over $\RCA_0^*$ with respect to $\Sigma_1$~sentences.
\end{thm}

We note that the speedup of $\RT^2_2$ over $\RCA^*_0$ can in fact be witnessed by $\Delta_0(\exp)$ sentences,
though the proof involves a slightly larger amount of background than for the $\Sigma_1$ case.
The idea of the argument is explained in the remark after the proof of Theorem \ref{thm:speedup}.

The basic reason why Theorem \ref{thm:speedup} holds is expressed in the following lemma.
Once the lemma is proved, the upper and lower bounds used to derive the theorem
are obtained by more or less standard techniques described e.g.~in the survey \cite{incoll:pflen}.

\begin{lem}\label{lem:size-closure}
$\RCA_0^* + \RT^2_2$ proves the following statement: ``for every $k$, if each infinite set has a finite subset with $k$ elements, then each infinite set has a finite subset with $2^k$ elements''.
\end{lem}

\begin{proof}
Clearly, it is enough to prove the statement with $2^{k/2}$ substituted for $2^k$.
Working in $\RCA_0^*+\RT^2_2$, let $k$ be such that the infinite set $A$ does not have a subset with $2^{k/2}$ elements. W.l.o.g., we may assume that $0 \in A$. 
We will use the  fact that there is a $2$-colouring of $[2^{k/2}]^2$ with no homogeneous set of size $k$,
which has a well-known probabilistic proof that easily formalizes in $\RCA_0^*$,
to define a $2$-colouring of $[\IN]^2$ such that
every infinite homogeneous sets for the new colouring fails to have a $k$-element subset.

So, let $f\colon [2^{k/2}]^2 \to 2$ have no homogeneous set of size $k$.
Let $\{a_i : i \in I\}$, for some cut $I$, be an increasing enumeration of $A$. Note that $2^{k/2} > I$.
Define $g: [\IN]^2 \to 2$ as follows:
\[g(x,y) =
\begin{cases}
b & \textrm{if there are } i < j \textrm{ s.t. } x \in [a_i, a_{i+1}), y \in [a_j, a_{j+1}), f(i,j) = b, \\
0 & \textrm{if there is } i  \textrm{ s.t. } x,y \in [a_i, a_{i+1}).
\end{cases}
\]

Let $H$ be an infinite $g$-homogeneous set.
By possibly thinning out $H$, we may assume w.l.o.g.~that for each $i \in I$, there is at most one $x \in H \cap [a_i, a_{i+1})$.
This means that for any two $x,y \in H$, the value $g(x,y)$ is determined by the first clause of the definition of $g$.

Let $x_1,\ldots,x_\ell$ be the first $\ell$ elements of $H$, and let $i_1,\ldots,i_\ell$ be such that $x_j \in [a_{i_j}, a_{i_j + 1})$.
By our choice of $H$ and the definition of $g$, we see that $f(i_j,i_{j'}) = g(x_j, x_{j'})$, so the set
$\{i_1,\ldots,i_\ell\}$ is homogeneous for $f$.  So $\ell<k$ by our choice of~$f$.
\end{proof}

Consider the simple variant of the iterated exponential function defined by  $2_0 = 1$ and $2_{n+1} = 2^{2_n}$ for $n\in \IN$.
To prove Theorem \ref{thm:speedup}, we will make use of a family of sentences stating the existence of $2_{2_{2_n}}$, for $n$ a natural number.
The relation $y = 2_x$ has a $\Delta_0$ definition which is well-behaved in $\ind\Delta_0$
 --- see \cite[\S3.1]{art:sharpMcAloon},
or see \cite[Chapt.~V.3(c)]{book:hajek+pudlak} for a similar idea applied to the more difficult case of $y = 2^x$).
 Thus, $\exists y\, (y = 2_{2_{2_n}})$ can be expressed as a $\Sigma_1$ sentence of size $O(\log n)$,
for instance, by substituting the numeral for $n$ into the fixed formula $\exists y\, (y = 2_{2_{2_x}})$.

\begin{lem}\label{lem:short-proof-supexp3}
In $\RCA^*_0 + \RT^2_2$, the sentence $\exists y\, (y = 2_{2_{2_n}})$ has a proof of size polynomial in $n$.
\end{lem}
\begin{proof}
Consider the definable set
\[I = \{k : \text{every infinite set has a finite subset with $2_k$ elements}\}.\]
In the absence of $\Sigma^0_1$ induction, this is a proper initial segment of $\IN$ and, in general, not~a set (i.e.~a second-order object).
However, Lemma \ref{lem:size-closure} implies that $I$ is a definable cut provably in $\RCA^*_0 + \RT^2_2$.

Consider also the set
\[A = \{\ell: \exists k \, (\ell = 2_k)\}.\]
Provably in $\RCA^*_0$, this is indeed a set,
because the $\exists k$ quantifier can be bounded by $\ell$.
By the totality of the exponential function, $A$ is infinite.
Moreover, since $k \mapsto 2_k$ is an increasing monotone
operation with a $\Delta_0$-definable graph,
$A$ has a subset with $k$ elements exactly for those $k$ for which $2_k$ exists.
It follows that for each $k$, if $I(k)$, then $2_{2_k}$ exists.

Now fix a natural number $n$.
By \cite[Theorem 3.4.1]{incoll:pflen}, there are formulas $I_0, I_1,\ldots, I_n$
such that $I_0$ is $I$ and for each $j<n$,
the theory $\RCA^*_0 + \RT^2_2$ has a $\mathrm{poly}(n)$-size proof
that $I_{j+1}$ is a cut and that $\forall k\,(I_{j+1}(k) \then I_j(2^k))$.
By putting these proofs together, we get a $\mathrm{poly}(n)$-size proof of $I(2_n)$
and hence of the statement that $2_{2_{2_n}}$ exists.
\end{proof}

\begin{lem}\label{lem:no-short-proof-supexp3}
The size of the smallest $\RCA^*_0$ proof of $\exists y\, (y = 2_{2_{2_n}})$
grows faster than any elementary recursive function of $n$.
\end{lem}

\begin{proof}
In view of Theorem~\ref{thm:IBWKL} in the next section, it suffices to show that the size of the shortest $\ind\Delta_0 + \exp$ proof of $\exists y\, (y = 2_{2_{2_n}})$ grows nonelementarily in $n$.
To this end, it is clearly enough to show the existence of a polynomial $p$ such that, if the sentence $\exists z\, \varphi(z)$ for a bounded formula
$\varphi$ has a proof in $\ind\Delta_0 + \exp$ of size $n$, then $\exle{z}{2_{2_{p(n)}}}\varphi(z)$ holds in $\omega$.

So, assume that $\ind\Delta_0 + \exp$ has a size-$n$ proof of $\exists z\, \varphi(z)$ for $\varphi$ bounded.
We may assume that $\ind\Delta_0 + \exp$ is finitely axiomatized, because
it has a finite axiomatization that polynomially simulates the more usual
one in which $\Delta_0$ induction is a scheme.
By standard upper bounds on the cut elimination procedure
and its connection to Herbrand's theorem \cite[Section 5]{incoll:pflen},
there is a fixed polynomial $p$ such that some Herbrand disjunction for $\neg (\ind\Delta_0 + \exp) \lor \exists z\, \varphi(z)$ of size at most $2_{p(n)}$
is logically valid.

In more detail, what this means is as follows. Consider the conjunction of:
\begin{itemize}
\item $\forall v_1\, \ldots \forall v_k \, \delta(v_1,\ldots,v_k)$, a statement that induction holds for finitely many fixed $\Delta_0$ formulas with parameters among the $\bar v$;
we may assume w.l.o.g.\ that all the quantifiers in $\delta$ are bounded by one of the $v_i$,
\item $\forall x\, \exists y\, (y=2^x)$, with $y = 2^x$ a $\Delta_0$ formula in the language of first-order arithmetic,
\item $\forall z\, \neg \varphi(z)$.
\end{itemize}
Put this conjunction in prenex normal form and skolemize it, obtaining the sentence $\forall \bar v \, \forall x\, \forall z\, \forall \bar w\, \xi(\bar v, x, z, \bar w)$, where the $\forall \bar w$ quantifiers are bounded and $\xi$ is quantifier-free (but contains the Skolem function symbols corresponding to the original existential quantifiers). Then, for some $s$, there is a list of closed terms
\[t_{v_1,1},\ldots,t_{v_k,1},t_{x,1},t_{z,1},t_{w_1,1},\ldots,t_{w_\ell,1}, \ldots, t_{v_1,s},\ldots,t_{v_k,s},t_{x,s},t_{z,s},t_{w_1,s},\ldots,t_{w_\ell,s}\]
(of the arithmetical language extended by the Skolem functions) such that the conjunction
\[\Xi \defeq \bigwwedge_{i=1}^s\xi(t_{\bar v,i},t_{x,i},t_{z,i},t_{\bar w,i})\]
has size at most $2_{p(n)}$ and is unsatisfiable.

We now interpret the function symbols appearing in $\Xi$ as operations on $\omega$.
The symbols $+,\cdot,0,1$ are interpreted in the usual way.
If $f(\ldots,x,\ldots)$ is the Skolem function symbol corresponding to the $\exists y$ quantifier
($f$ may have more arguments as a result of the move to prenex normal form),
then we put $f(\ldots,m,\ldots) = 2^m$.
If $f$ is the Skolem function symbol corresponding to a bounded existential quantifier $\exists u$
in front of a subformula $\psi$ of $\delta$, $\neg \varphi$, or $y = 2^x$,
we let $f$ pick out the smallest witness to $\exists u \, \psi$ on those arguments for which there is such a witness;
otherwise, $f$ returns $0$.

Since each function interpreting a symbol in $\Xi$ increases its arguments at most exponentially,
and $\Xi$ has size at most $2_{p(n)}$, it follows that under our interpretation
each term in $\Xi$ has a value bounded by $2_{2_{p(n)}}$.
Since we interpreted the Skolem function symbols as actual Skolem functions, and $\ind \Delta_0 + \exp$ is a true theory,
each substitution instance of the skolemizations of $\forall v_1\, \ldots \forall v_k \, \delta(v_1,\ldots,v_k)$ and $\forall x\, \exists y\, (y=2^x)$
appearing in $\Xi$ is true under our interpretation.
However, $\Xi$ is unsatisfiable, which implies that some substitution instance of the skolemization of $\forall z\, \neg \varphi(z)$ must be false.
Due to the way we defined our interpretation, this means that $\exle{z}{2_{2_{p(n)}}}\varphi(z)$ holds.
\end{proof}

\begin{proof}[Proof of Theorem \ref{thm:speedup}]
Immediate from Lemmas \ref{lem:short-proof-supexp3} and \ref{lem:no-short-proof-supexp3}.
\end{proof}

\begin{remark}
As already mentioned, it is possible to witness the speedup of $\RCA^*_0 + \RT^2_2$ over $\RCA^*_0$ by a family of $\Delta_0(\exp)$ sentences.
The sentences in question take the form $\Con_{2_n}(\ind\Delta_0 + \exp)$, where $\Con_x(T)$ is a formula stating in a natural way that there is no proof of contradiction in the theory $T$ of size at most $x$. Clearly, $\Con_{2_n}(\ind\Delta_0 + \exp)$ can be expressed by a $\Delta_0(\exp)$ sentence of size polynomial in $n$.

The so-called finitistic G\"odel theorem, proved independently by Friedman and Pudl\'ak (see \cite[Theorem 6.3.2]{incoll:pflen} and the references therein), implies that for a sufficiently strong finitely axiomatized $T$, the size of the smallest proof of $\Con_{n}(T)$ in $T$ is $n^{\Omega(1)}$. As a consequence, the size of the smallest proof of $\Con_{2_n}(\ind\Delta_0 + \exp)$ in $\ind\Delta_0 + \exp$ is $(2_n)^{\Omega(1)}$.
By Theorem~\ref{thm:IBWKL},
a proof of this sentence in $\RCA^*_0$ also requires size  $(2_n)^{\Omega(1)}$.

On the other hand, essentially by formalizing the proof of Lemma \ref{lem:no-short-proof-supexp3} but with $\forall z\, \neg \varphi(z)$ replaced by $0=0$, one can prove in $\ind\Delta_0 + \exp$ that, for a certain fixed polynomial $p$, the existence of $2_{2_{p(x)}}$ implies $\Con_x(\ind\Delta_0 + \exp)$. By Lemma~\ref{lem:short-proof-supexp3}, $\RCA^*_0 + \RT^2_2$ can prove the existence of $2_{2_{2_{p(n)}}}$, a number greater than
$2_{2_{p(2_n)}}$, in size $\mathrm{poly}(n)$. Thus, it can also prove $\Con_{2_n}(\ind\Delta_0 + \exp)$ in size $\mathrm{poly}(n)$.
\end{remark}

\section{Induction versus collection}%

Since $\RT^2_2$ proves the collection principle $\bd\Sigma_2$,  Corollary \ref{cor:simulation-rt22-is1}
implies that $\ind\Sigma_1$ polynomially simulates $\bd\Sigma_2$ w.r.t.~proofs of $\Pi_3$ sentences.
In this context, it is natural to ask whether the $\Pi_{n+2}$-conservativity of $\bd\Sigma_{n+1}$ over
$\ind\Sigma_n$ is also witnessed by a polynomial simulation for $n \neq 1$. A related problem
was raised by Clote et al.~\cite[page 216]{art:formconserv}: is the conservativity provable in bounded arithmetic?
Note that if a $\Pi_2$ sentence of the form $\forall x\, \exists y \, \delta(x,y)$,
such as a conservativity statement, is provable in bounded
arithmetic, then the least witness for the $\exists y$ quantifier can be
bounded by a polynomial-time function of $x$.

In this section, we present, for each $n\in\IN$, a polynomial forcing interpretation~$\tau(n)$
of $\bd\Sigma_{n+1}+\exp$ in $\ind\Sigma_n+\exp$ that is polynomially $\Pi_{n+2}$-reflecting.
The existence of these forcing interpretations, of course, implies that
$\ind\Sigma_n$ polynomially simulates $\bd\Sigma_{n+1}$ w.r.t.~$\Pi_{n+2}$ sentences
whenever $n \ge 1$; for $n = 0$, we get
a polynomial simulation of $\bd\Sigma_1 + \exp$ by $\ind\Delta_0 + \exp$.
Our interpretations can be formalized in (a fragment of)
bounded arithmetic, which partially solves the problem from~\cite{art:formconserv} as well:
completely for $n \ge 1$, and over $\exp$ for $n = 0$.

The only role of $\exp$ is to provide us with the usual universal $\Sigma_{n+1}$ formula
$\Sigma_{n+1}\hyp\Sat(v,x)$. Recall that $\ind\Delta_0+\exp$ proves $\theta(\bar x) \nsc \Sigma_{n+1}\hyp\Sat(\nmrl{\theta},\langle\bar x\rangle)$ for $\theta \in \Sigma_{n+1}$,
where $\nmrl{\theta}$ denotes the numeral for (the code/G\"odel number of) 
$\theta$ and $\langle \bar x\rangle$ denotes the code for the finite sequence $\langle x_{1},\dots,x_{l}\rangle$.
It also proves that $\Sigma_{n+1}\hyp\Sat(v,x)$ satisfies the usual recursive conditions for a satisfaction relation restricted to $\Sigma_{n+1}$ formulas;
cf.~\cite[Chapts.~I.1(d), V.5(b)]{book:hajek+pudlak}.

\begin{remark}
The question from~\cite{art:formconserv} as stated concerned provability in $\ind\Delta_0 + \Omega_1$,
which coincides with the $\Lone$-consequences of Buss' theory $S_2$.
The related Question 35 from the Clote-Kraj\'i\v cek list of open questions~\cite{incoll:openprob} ---
whether the stronger theory $\ind\Delta_0 + \exp$ proves the mere
equiconsistency of $\bd\Sigma_{n+1}$ and $\ind\Sigma_n$ ---
was already answered positively in H\'ajek's paper~\cite[Section~2]{incoll:hajek/interpret}
in the same volume where the list appeared.
\end{remark}

\begin{remark}
While we were writing up our paper, we learnt that
 Fedor Pakhomov [private communication]
 independently proved that $\ind\Sigma_n$ polynomially simulates $\bd\Sigma_{n+1}$ w.r.t.~$\Pi_{n+2}$
 (and that bounded arithmetic proves this)
 by a different argument that avoids the need for $\exp$ in the case $n=0$.
Moreover, in an email discussion with us,
 Pakhomov found a parametric interpretation
  of $\bd\Sigma_{n+1}+\exp+\neg\sigma$
  in $\ind\Sigma_n+\exp+\neg\sigma$
   for each $\Pi_{n+2}$~sentence $\sigma$, at least for $n>0$.
It would be interesting to know
 whether his interpretation can be made independent
  of the $\Pi_{n+2}$~sentence in question, and
 whether the use of parameters can be avoided.
\end{remark}

The model-theoretic idea behind our forcing interpretations is simple
and appeared implicitly already in~\cite[Lemma~12]{art:formconserv}:
given $M\models\ind\Sigma_n+\exp$,
 if we construct a non-cofinal extension $K\elemext_{\Sigma_{n+1}}M$,
   then $M\elemsub_{\Sigma_{n+1}}\sup_KM\models\bd\Sigma_{n+1}+\exp$.
If the aim is only to show the $\Pi_{n+2}$-conservativity
 of $\bd\Sigma_{n+1}+\exp$ over $\ind\Sigma_n+\exp$,
  one can build the extension~$K$ by any method,
   but to get a forcing interpretation a forcing construction is needed.
The forcing construction we choose is that of
  a generic $\Sigma_{n+1}$ ultrapower~\cite[pages~181f.]{art:subarith-ultrapower};
   cf.~the proof of Theorem~B in~\cite{Paris-Kirby}.
In particular, the construction employs a maximal filter~$\VF$
  in the lattice consisting of the $\Sigma_{n+1}$-definable subsets of $M$.
Elements of the $\Sigma_{n+1}$ ultrapower~$K$ are represented by
  $\Sigma_{n+1}$-definable partial functions $M\parto M$
   whose domain is an element of~$\VF$.
To ensure that the ultrapower is a non-cofinal extension,
 we require every element of~$\VF$ to be cofinal in~$M$.
It would be possible to first define a forcing relation $\forces_{\tau(n)}'$ for the ultrapower~$K$
    and then define a second forcing relation $\forces_{\tau(n)}$
     for the truncation~$\sup_KM$ in terms of $\forces_{\tau(n)}'$.
To save some notation, we introduce the second one directly.

An alternative forcing interpretation, much closer in spirit to the one used to prove Theorem \ref{thm:simulation-rt22-rca0}
and therefore perhaps easier to follow for the reader acquainted with Section \ref{sec:ramsey},
is mentioned in a remark after the proof of Theorem~\ref{thm:IBfint}.
However, that interpretation only works for $n > 0$.

In the remainder of this section, $\exists^\infty x$ (``for infinitely many $x$'' or ``for cofinally many $x$'')
stands for $\forall y\, \exists x{\ge}y$ and $\forall^\infty x$ (``for almost all $x$'') stands for $\exists y\, \forall x{\ge}y$.

For $n \in \IN$, we will define a forcing translation~\defm{$\tau(n)$}
from the language of first-order arithmetic to itself. The definition of the set
of conditions is relatively straightforward.

\begin{defn}
Let $n\in\IN$, and let $x\in\Ext_{n+1}(s)$ stand for $\Sigma_{n+1}\hyp\Sat(s,\langle x\rangle)$.
Then a (code for a) $\Sigma_{n+1}$ formula with one free variable $s$ is in $\Cond[\tau(n)]$ if $\excf x{(x\in\Ext_{n+1}(s))}$. 
For two conditions $s, s'$, we let $s'\extendseq_{\tau(n)}s$ if $\Ext_{n+1}(s')\subseteq\Ext_{n+1}(s)$.
\end{defn}

%

In other words, conditions are infinite $\Sigma_{n+1}$-definable sets. 
Our names will be (codes for) $\Sigma_{n+1}$ formulas with two variables. 
The definition of $s\forces_{\tau(n)}v\defd$ is somewhat subtle,
but the idea is that names are to be viewed as 
$\Sigma_{n+1}$-definable functions whose ranges are bounded.
To capture this intuition, we introduce some auxiliary concepts.
We will write $\val v x\defd$ if $\exi y{\Sigma_{n+1}\hyp\Sat(v,\langle x,y\rangle)}$,
and simply write $\val v x$ for the unique such $y$.
We write $s\forcesp_{\tau(n)} v\defd$ if there exists some $d$ such that
$\fain x{\Ext_{n+1}(s)}{(\val v x\defd \land \val v x \le d)}$.
Note that if $v_1,\ldots,v_\ell$ are names, then under the assumption $\val {v_1} x\defd,\ldots,\val {v_\ell} x \defd$
the assertion $\theta(\val {\bar v} x)$ can be formulated either as 
\[\ex{y_1,y_2,\dots,y_\ell}{\Bigl(
    \bigwwedge_{i=1}^\ell \Sigma_{n+1}\hyp\Sat(v_i,\langle x,y_i\rangle)
    \wedge\theta(\bar y)
   \Bigr)}\]
or as
\[\fa{y_1,y_2,\dots,y_\ell}{\Bigl(
    \bigwwedge_{i=1}^\ell \Sigma_{n+1}\hyp\Sat(v_i,\langle x,y_i\rangle)
    \to\theta(\bar y)
   \Bigr)}.\] 
For $\theta \in \Sigma_n \cup \Pi_n$, these statements
are $\Sigma_{n+1}$ and $\Pi_{n+1}$, respectively. 
Below, we sometimes write $\bar v(x) \defd$ for the conjunction
of $\val {v_1} x\defd,\ldots,\val {v_\ell} x \defd$.

The actual definition of $s\forces_{\tau(n)}v\defd$ 
says that $s'\forcesp_{\tau(n)} v\defd$
happens densely below $s$.

\begin{defn}\label{defn:name-defd}
For a condition $s$ and a name $v$,
we say that $s\forces_{\tau(n)}v\defd$ if the following hold:
\begin{enumerate}[(i)]
\item $\aallin x{\Ext_{n+1}(s)}{\exists^{\le 1} y\, {\Sigma_{n+1}\hyp\Sat(v,\langle x,y\rangle)}}$,
\item $\exists d\, \aallin x{\Ext_{n+1}(s)}{\forall y\, {(\Sigma_{n+1}\hyp\Sat(v,\langle x,y\rangle) \then y \le d)}}$,
\item there is no $s' \extendseq_{\tau(n)}s$ such that $\aallin x{\Ext_{n+1}(s')}{\forall y\,\neg{\Sigma_{n+1}\hyp\Sat(v,\langle x,y\rangle)}}$.
\end{enumerate}


For a simple atomic formula $\alpha(v_1,\ldots,v_\ell)$, we say that
$s\forces_{\tau(n)}\alpha(\bar v)$ if $s\forces_{\tau(n)}\bar v\defd$ and
\[\aallin x{\Ext_{n+1}(s)}{{\Bigl(
    \bigwwedge_{i=1}^\ell v_i(x)\defd \then \alpha(v_1(x),\ldots,v_\ell(x))
   \Bigr)}}.\]

When there is no risk of ambiguity,
 we will often omit the subscript $\tau(n)$.
\end{defn}   


\begin{lem}\label{lem:ultrapower-fint}
For each $n  \in \IN$, the relations $\Cond_{\tau(n)}, \extendseq_{\tau(n)}, \forces_{\tau(n)}$ determine a forcing interpretation
of $\Lone$ in $\ind\Sigma_n + \exp$. Moreover, there is a polynomial-time procedure which, given $n\in\IN$ in unary and
an instance of \ref{item:fint/ex_Cond}--\ref{item:fint/density} for $\tau(n)$, outputs a proof of that instance in $\ind\Sigma_n + \exp$.
\end{lem} 

In this and later similar statements on the existence of polynomial-time procedures that output some proofs, 
the dependence of the running time on~$n$ rather than just on the formula to be proved cannot be avoided:
it comes for instance 
from the universal $\Sigma_{n+1}$~formula,
which is already used in the definition of $\Ext_{n+1}$, and thus of $\Cond[\tau(n)]$ and $\forces_{\tau(n)}$.

\begin{proof}
The only conditions that are not completely trivial to prove are \ref{item:fint/densedefd} and \ref{item:fint/density}.
Below, we think of $n \in \IN$ as fixed and explain how to prove \ref{item:fint/densedefd} and \ref{item:fint/density}
for $\tau(n)$ in $\ind \Sigma_n+\exp$. The verification that the proofs can be constructed in polynomial time is left to the reader.

Condition
\ref{item:fint/densedefd} states that if $s'\forces v\defd$ happens densely below $s$, then in fact $s\forces v\defd$. So assume that $s\not\forces v\defd$. If item (iii) from the definition of $s\forces v\defd$ is violated, this immediately gives some $s' \extendseq s$ such that $v(x)$ is undefined for each $x \in \Ext_{n+1}(s')$, so clearly $s'' \not \forces v\defd$ whenever
$s'' \extendseq s'$. If (i) is violated, let $s'$ be a (code for a) formula such that $x \in \Ext_{n+1}(s')$ iff 
$x \in \Ext_{n+1}(s') \land \exists^{\ge 2} y\, {\Sigma_{n+1}\hyp\Sat(v,\langle x,y\rangle)}$. Then $s' \in \Cond$, 
and any $s'' \extendseq s'$ violates (i) with respect to $v$. 
Finally, if (ii) is violated but (i) holds, first let $b$ be such that $\exists^{\le 1} y\, {\Sigma_{n+1}\hyp\Sat(v,\langle x,y\rangle)}$ holds for $x \in \Ext_{n+1}(s)$ with $x > b$. Let $s' \extendseq s$ be such that $x \in \Ext_{n+1}(s')$ iff $x \in \Ext_{n+1}(s) \land x > b$. 
The failure of (ii) means that for each $d$, there exists $x \in \Ext_{n+1}(s')$
with $x > d$ and $v(x)\defd \land v(x) > d$.
Let $H$ be the set of all finite sequences of the form $(x_i,y_i,w_i)_{i \le k}$ such that
for each $i \le k$, the triple $\langle x_i, y_i, w_i \rangle$ is the smallest one satisfying:
\begin{itemize}
\item $w_i$ is a witness for (the outermost existential quantifier block of) the $\Sigma_{n+1}$ statement $x_i \in \Ext_{n+1}(s') \land \Sigma_{n+1}\hyp\Sat(v,\langle x_i,y_i\rangle)$, 
\item $x_i > \langle x_j, y_j, w_j \rangle$ for each $j < i$,
\item $y_i > \langle x_j, y_j, w_j \rangle$ for each $j < i$.
\end{itemize}
It can be observed that the set $H$ is $\Sigma_{0}(\Sigma_{n})$-definable; in fact, in the presence
of $\bd\Sigma_n$ it can be defined by the conjunction of a $\Pi_n$ and a $\Sigma_n$ formula. 
Moreover, $H$ has no maximal element, by our assumption on the failure of (ii).
Thus, $H$ is cofinal by $\ind\Sigma_{n}$.
Now define $s''$ so that $x \in \Ext_{n+1}(s'')$ iff there is a finite sequence $(x_i,y_i,w_i)_{i \le k}$ in $H$ such that $x=x_{k}$.
Then $s''\in\Cond$, $s'' \extendseq s'$, and for each $d$, if $x \in \Ext_{n+1}(s'')$ and $x \ge d$,
then $v(x') > d$ for each $x' \in \Ext_{n+1}(s'')$ with $x' > x$. So, any $s''' \extendseq s''$ violates (ii) with respect to $v$.

Once \ref{item:fint/densedefd} is shown, the proof of \ref{item:fint/density} becomes simple.
If $s \not \forces \alpha(\bar v)$, then there are two cases to consider. Either $s \not \forces \bar v \defd$,
and then by the above argument $s' \forces \bar v \defd$ does not occur densely below $s$.
Otherwise, we have
\[\excfin x{\Ext_{n+1}(s)}{{\Bigl(
    \bigwwedge_{i=1}^\ell v_i(x)\defd \land \neg\alpha(v_1(x),\ldots,v_\ell(x))
   \Bigr)}}.\]
Let $b$ be such that $\exists^{\le 1} y\, {\Sigma_{n+1}\hyp\Sat(v,\langle x,y\rangle)}$ holds for $x \in \Ext_{n+1}(s)$ with $x > b$. Let $s' \extendseq s$ be such that $x \in \Ext_{n+1}(s')$ iff $x \in \Ext_{n+1}(s) \land x > b$.
Let $s'' \extendseq s'$ be such that 
\[x \in \Ext_{n+1}(s'') \nsc x \in \Ext_{n+1}(s') \land  
    \bigwwedge_{i=1}^\ell v_i(x)\defd \land \neg\alpha(v_1(x),\ldots,v_\ell(x)).\]
Then $s''' \not \forces \alpha(\bar v)$ whenever $s''' \extendseq s''$.    
\end{proof} 

\begin{lem}\label{lem:weak-strong-defd}
For any $s\in\Cond$ and names $v_{1}\dots, v_{\ell}$ such that $s\forces \bar v\defd$, there exists $\hat s\extendseq s$ such that: 
\begin{itemize}
\item $\hat s\forcesp \bar v$, 
\item $\aallin x{\Ext_{n+1}(s)}{(\bar v(x)\defd\to x\in\Wx{\hat s})}$,
\item $s\forces\theta(\bar v)\nsc \hat s\forces\theta(\bar v)$ for any $\lang_{1}$-formula $\theta$.
\end{itemize}
\end{lem}
\begin{proof}
Take any $s\in\Cond$ and $\bar v$
such that $s\forces\bar v\defd$.
By (i) and (ii) of the definition of $s\forces\bar v\defd$, let $b,d$ be such that 
for each $x \in \Ext_{n+1}(s)$ with $x > b$ and each $i=1,\dots,\ell$,
\[\exists^{\le 1} y\, {\Sigma_{n+1}\hyp\Sat(v_{i},\langle x,y\rangle)} \land 
\forall y\, (\Sigma_{n+1}\hyp\Sat(v_{i},\langle x,y\rangle) \then y \le d)\] holds.
Let $\hat s\extendseq s$ be such that $x\in\Wx{\hat s}\nsc x\in\Wx{s}\wedge x>b\wedge \bar v(x)\defd$.
Then we have $\aallin x{\Ext_{n+1}(s)}{(\bar v(x)\defd\to x\in\Wx{\hat s})}$.

We show $s\forces\theta(\bar v)\nsc \hat s\forces\theta(\bar v)$ for any $\lang_{1}$-formula $\theta$.
The implication from left to right is immediate from Lemma~\ref{lem:monodense}. 
To show the converse implication, 
assume that $\hat s\forces\theta(\bar v)$ for an $\lang_{1}$-formula $\theta$.
By Lemma~\ref{lem:monodense}, it is enough to show that $\theta(\bar v)$ is forced dense below $s$.
Let $s'\extendseq s$. Take $s''\extendseq s'$ so that $x\in\Wx{s''}\nsc x\in\Wx{s'}\wedge x>b\wedge \bar v(x)\defd$.
Then, $s''\extendseq \hat s$, and hence $s''\forces\theta(\bar v)$.
\end{proof}


It is known that $\Sigma_{n+1}$~ultrapowers of models of~$\ind\Sigma_n$
 satisfy \L o\'s's Theorem
  for $\Sigma_{n+1}$ and $\Pi_{n+1}$~formulas~\cite[Lemma~5.13]{art:subarith-ultrapower}.
This ensures $\Sigma_{n+2}$-elementarity between
 the base model~$M$ and the ultrapower~$K$,
  although, as remarked earlier,  $\Sigma_{n+1}$-elementarity is already enough to run our argument.
We transform these ideas into a syntactic proof in the following.

\begin{lem}\label{prop:los}
There is a polynomial-time procedure which,
 given $n\in\IN$ in unary and a formula $\theta(y_1,y_2,\dots,y_\ell)$ in~$\Sigma_n\cup\Pi_n$,
 outputs a proof of
 the following in $\ind\Sigma_n+\exp$:
\begin{itemize}
\item[] for any $s\in\Cond$ and names $v_{1}\dots, v_{\ell}$ such that $s\forces_{\tau(n)} \bar v\defd$, \\
$s\forces_{\tau(n)}\theta(\bar v)\nsc\aallin x{\Ext_{n+1}(s)}{ \left(\val{\bar v}x \defd \then \theta(\val{\bar v}x) \right)}.$
\end{itemize}
\end{lem}

\begin{proof}
Like the usual proof of \L o\'s's Theorem
  for $\Sigma_{n+1}$~ultrapowers~\cite[Lemma~5.13]{art:subarith-ultrapower},
   the construction splits into two parts.
\begin{enumerate}
\item For $\Delta_0$~formulas~$\theta$, we build the required proof
  by induction on the structure of~$\theta$.\label{st:los/D0}
\item Beyond that, we build the proof using induction
  on the number of unbounded quantifiers in~$\theta$. \label{st:los/q}
\end{enumerate}


As in the proof of Lemma \ref{lem:ultrapower-fint}, we think of $n \in \IN$ as fixed,
and leave it to the reader to verify that the proofs in $\ind\Sigma_n+\exp$
described below can be found in time polynomial in the given parameters.

In the construction for~\partref{st:los/D0}, the base step is for an atomic formula $\theta(\bar v)$. 
If $\theta$ is simple atomic, the required equivalence follows immediately from the definition of $\forces_{\tau(n)}$.
For a non-simple atomic $\theta$, the argument is similar to the one in Lemma \ref{lem:D0elem}.
By induction on the structure of a term $t(\bar u)$
($\bar u$ a subtuple of $\bar v)$ appearing in $\theta$,
we build a proof that, if $w$ is a name, then
$s \forces t(\bar u) = w$ if and only if the following happens:
$s \forces \bar u \defd$, $s \forces w \defd$,
and for almost all $x \in \Ext_{n+1}(s)$, if $\val {\bar u}x\defd$ and $\val w x\defd$, then $t(\val {\bar u}x) = \val w x$.
Once this is proved for each term in $\theta$, the required equivalence is proved using
\ref{item:ftr/forall}, \ref{item:ftr/then}, and the definition of $\forces_{\tau(n)}$.

The inductive steps of the construction for~\partref{st:los/D0} are more or less based
on the usual proof of \L o\'s's Theorem with the
satisfaction relation for the $\Sigma_{n+1}$ ultrapower replaced by the forcing relation.
As a demonstration, we show the case for negation.
Lemma~\ref{lem:weak-strong-defd}
lets us restrict attention to the case when $s\forcesp \bar v\defd$, 
which in particular implies $\val{\bar v}x\defd$ for all $x \in \Ext_{n+1}(s)$.

Suppose $\theta$ is $\neg \eta$ for $\eta \in\Delta_0$.
The inductive assumption gives us a proof
that if $s$ and $\bar v$ are such that $s \forcesp \bar v\defd$,
then $s \forces \eta(\bar v)$ exactly if for almost all $x \in \Ext_{n+1}(s)$
the values of the $v_i$'s at $x$ satisfy $\eta$.

Take any $s\in\Cond$ and any~$\bar e$
such that $s\forcesp\bar e\defd$. 
If $s\nforces\neg\eta(\bar e)$,
 then \ref{item:ftr/neg}~gives us $s'\extendseq s$
  such that $s'\forces\eta(\bar e)$, and
   so, by the proof from the induction hypothesis,
    $\aallin x{\Ext_{n+1}(s')}\, \eta(\val{\bar e} x)$,
hence also
       $\neg \aallin x{\Ext_{n+1}(s)}\, \neg\eta(\val{\bar e} x)$.

Conversely, suppose $s\forces\neg\eta(\bar e)$.
By~\ref{item:ftr/neg} and the proof from the induction hypothesis,
 \begin{equation*}
  \faexteq{s'}{s}{
   \excfin x{\Ext_{n+1}(s')}\, \neg\eta(\val{\bar e} x)
}.
 \end{equation*}
Take a (code for a) $\Sigma_{n+1}$~formula $s'$ so that $x\in \Wx {s'}\leftrightarrow x\in \Wx s \wedge  \eta(\val{\bar e} x)$. By the above, $s'$ cannot be a member of $\Cond$, and thus $\Wx {s'}$ is bounded.
Hence, we have
  $\neg\excfin x{\Ext_{n+1}(s)}{\eta(\val{\bar e} x)}$.
So, given that $s \forcesp \bar e \defd$, we have
  $\aallin x{\Ext_{n+1}(s)}{\neg\eta(\val{\bar e} x)}$.

If $n=0$, then \partref{st:los/D0}~already suffices.
So suppose $n>0$.
In the construction for~\partref{st:los/q},
 we also imitate the usual proof of \L o\'s's Theorem,
  but additionally we need to take care of the fact that
   our anticipated ultrapower is truncated.
The bounds involved will be provided
by strong $\Sigma_n$~collection, which is the principle that given a $\Sigma_n$ formula $\varphi(x,y)$
and a bound $a$, there is a bound $b$ such that for all $x<a$, if any $y$ satisfying $\varphi(x,y)$ exists,
then some such $y$ may be found below $b$. It is well-known that for $n > 0$ this is equivalent to
$\ind\Sigma_n$, cf.~\cite[Theorem~I.2.23(a)]{book:hajek+pudlak}.
The only part of~\partref{st:los/q} that requires special attention is
 the $\then$~direction for the universal quantifier.
 Again, by Lemma~\ref{lem:weak-strong-defd} we only need to consider the case when $s\forcesp \bar v\defd$.

So, suppose $\theta\in \Sigma_{n}\cup\Pi_{n}$ is of the form $\fa{y}{\eta(\bar x,y)}$.
Note that $\eta\in \Pi_{n}$ in such a case.
Work over $\ind\Sigma_n+\exp$.
The inductive assumption gives us a proof that
if $s, \bar v, w$ are such that $s \forcesp \bar v \defd, w \defd$,
then $s \forces \eta(\bar v, w)$ exactly if for almost all $x \in \Ext_{n+1}(s)$,
it holds that $\eta(\bar v(x), w(x))$.

Take any $s\in\Cond$ and any~$\bar e$
 such that $s\forcesp\bar e\defd$.
  Assume
 \begin{equation*}
  \excfin x{\Ext_{n+1}(s)}{\neg\fa{y}{\eta(\val{\bar e} x,y)}
}.
 \end{equation*}
Take a (code for a) $\Sigma_{n+1}$~formula $s'$ so that
\[x\in \Wx {s'}\leftrightarrow x\in \Wx s\wedge\ex{y}{\neg\eta(\val{\bar e} x,y)}.\]
Then, by our assumption, $s'\in\Cond$ and $s'\extendseq s$.
Take a (code for a) $\Sigma_{n+1}$~formula $e'$
 with free variables $x, z$
 so that
  \begin{equation*}
\Sigma_{n+1}\hyp\Sat(e',\langle x,z\rangle)
\leftrightarrow
   {x\in \Wx {s'}
    \wedge\neg\eta(\val{\bar e}x,z)
    \wedge\falt{z'}{z}\eta(\val{\bar e}x,z')
}.
  \end{equation*}
Since $\ind\Sigma_n$ holds and $s'\forcesp \bar e$, we know
 \begin{math}
\fain x{\Ext_{n+1}(s')}{
   \exi z{\Sigma_{n+1}\hyp\Sat(e',\langle x,z\rangle)}
  }
 \end{math}.

Recall that $s\forcesp\bar e\defd$ and $s'\extendseq s$.
Find~$d$ such that
 \begin{equation*}
  \fain x{\Ext_{n+1}(s')}{\bigwwedge_{i=1}^\ell 
  e_i(x) \le d
   }.
 \end{equation*}
Use strong $\Sigma_{n}$~collection
 to obtain~$d'$ such that
 \begin{equation*}
  \fale{\bar y}d{\bigl(
   \ex{z}{\neg\eta(\bar y,z)}
   \then\exle{z}{d'}{\neg\eta(\bar y,z)}
  \bigr)},
 \end{equation*}
which ensures
 \begin{equation*}
\fain x{\Ext_{n+1}(s')}{\exle z{d'}{
   \Sigma_{n+1}\hyp\Sat(e',\langle x,z\rangle)
  }}.
 \end{equation*}
Thus $s'\forcesp\bar e'\defd$, and the proof from the induction hypothesis gives us
  $s'\forces\neg\eta(\bar e,\bar e')$.
Hence \ref{item:ftr/forall}
 tells us $s\nforces\fa{y}{\eta(\bar e,y)}$.
\end{proof}

\begin{remark} For $n > 0$, the argument in the proof above
relies on induction axioms in two places.
 Speaking in model-theoretic terms,
  first we use at most $\ind\Sigma_{n-1}$ to guarantee the existence
   of sufficiently many Skolem functions when proving \L o\'s's Theorem
   for the ultrapower~$K$, and
 then we use strong $\Sigma_n$~collection to pass on
  the elementarity from~$K$ to~$\sup_KM$. 
For $n=0$, the latter application is not needed 
  (and we do not have strong collection anyway),
 but we use full $\ind\Sigma_n$ to get the \L o\'s Theorem we want.
\end{remark}

We can now complete the argument that $\tau(n)$ 
is a forcing interpretation of $\bd\Sigma_{n+1} + \exp$ in $\ind\Sigma_n + \exp$.
In the usual argument proving $\Pi_{n+1}$ conservativity of $\bd\Sigma_{n+1}$ of $\ind\Sigma_n$~\cite[page~229]{book:hajek+pudlak} it is easy to deduce from the analogue of Lemma~\ref{prop:los}
  that the truncated ultrapower~$\sup_KM$ is a $\Sigma_{n+1}$-elementary extension of the base model~$M$.
In the model-theoretic setting, this directly implies $\sup_KM\models\ind\Delta_0$.
When formulated in terms of forcing interpretations,
 elementarity becomes reflection,
  but being reflecting in the sense of Definition~\ref{def:reflects}
   is not enough to ensure that \emph{every} condition forces~$\ind\Delta_0$,
    which is required by~\ref{item:fint/thy}.
So we strengthen the notion of reflection in the statement below.
From this, one quickly derives~$\exp$
 using the cofinality of~$M$ inside $\sup_KM$ and the $\ind\Delta_0$-provable monotonicity of the exponential function.

\begin{thm}\label{thm:IBfint}
For every $n\in\IN$,
 the forcing translation $\tau(n)$ is a polynomial forcing interpretation
  of $\bd\Sigma_{n+1}+\exp$ in $\ind\Sigma_n+\exp$
   that is polynomially $\Pi_{n+2}$-reflecting.

In fact the $\tau(n)$ are polynomially $\Pi_{n+2}$-reflecting in the following strengthened
and uniform sense:
there is a single polynomial-time procedure which, given $n\in\IN$ in unary
and a $\Pi_{n+2}$~sentence~$\gamma$,
 outputs a proof of
 \begin{equation*}
  \exin s{\Cond[\tau(n)]}{(s\forces_{\tau(n)}\gamma)}\then\gamma
 \end{equation*} in $\ind\Sigma_n+\exp$.
\end{thm}

\begin{proof}
As in earlier arguments in this section, we think of $n \in \IN$ as fixed
and leave it to the reader to check that the running time of a procedure constructing
the proofs in $\ind\Sigma_n + \exp$ described below
can be bounded by a polynomial in $n$ and $|\gamma|$.
We only have to consider \ref{item:fint/theory} and the strengthened reflection property,
as the other conditions from the definition of forcing interpretation 
have been dealt with in Lemma \ref{lem:ultrapower-fint}.

\paragraph{Strengthened reflection.}
Let $\gamma$ be a $\Pi_{n+2}$ sentence of the form $\fa{\bar x}{\ex{\bar y}{\theta(\bar x,\bar y)}}$, where $\theta(x_1,x_2,\dots,x_k,y_1,y_2,\dots,y_\ell)$ is $\Pi_n$.
Let $s$
 be such that $s\forces\fa{\bar x}{\ex{\bar y}{\theta(\bar x,\bar y)}}$.
Take any $a_1,a_2,\dots,a_k$.
Then, as one can verify using~\ref{item:ftr/forall} and the definition of $\forces_{\tau(n)}$,
 we have $s\forces\ex{\bar y}{\theta(\bar{\cname a},\bar y)}$,
  where each $\cname a_i$ is the (code for the)
$\Sigma_{n+1}$~formula $x = x \wedge y=\nmrl{a_i}$, which defines the graph of the constant function with value $a_{i}$.
From this, we deduce that
 \begin{math}
  \faexteq{s'}s{\exexteq{s''}{s'}{\ex{\bar v}{\bigl(
   s''\forces\theta(\bar{\cname a},\bar v)
  \bigr)}}}
 \end{math}.
We then apply Lemma~\ref{prop:los} to get
 \begin{equation*}
  \faexteq{s'}s{\exexteq{s''}{s'}{\ex{\bar v}{\big(s'' \forces \bar v \defd \land 
   \aallin x{\Ext_{n+1}(s'')}{(\val{\bar v}x \defd \then (\theta(\bar a,\val{\bar v}x))\big)
}
  }}}.
 \end{equation*}
In particular, we know $\ex{\bar y}{\theta(\bar a,\bar y)}$ holds.

\paragraph{Interpretation of $\ind\Delta_0+\exp$.}
The $\ind\Delta_0$~part follows easily
 from strengthened reflection for $\Sigma_1$~sentences.
For $\exp$, simply prove that for any $s\in\Cond$ and any $e$,
 if $s\forces e\defd$,
  then $s\forces e'=2^e$, where $e'$ is a (code for a) $\Sigma_{n+1}$~formula such that
   \begin{equation*}
     \Sigma_{n+1}\hyp\Sat(e',\langle x,z\rangle)
     \leftrightarrow
         \ex z{\bigl(
     \Sigma_{n+1}\hyp\Sat(e,\langle x,y\rangle)
     \wedge y=2^z
    \bigr)}.
   \end{equation*}

\paragraph{Interpretation of $\bd\Sigma_{n+1}$.}
By induction on~$m \le n$, we construct proofs of
 \begin{equation*}
  \fain s\Cond{(s\forces\bd\Pi_m)}.
 \end{equation*} 
This is sufficient by Proposition~\ref{prop:logic+}
  and the well-known equivalence between $\bd\Pi_n$ and $\bd\Sigma_{n+1}$.
Take any $m\leq n$.
In the case when $m\not=0$, suppose that
 we already have a proof of $\fain s\Cond{(s\forces\bd\Pi_{m-1})}$.

Consider a $\Pi_m$~formula $\theta(x,y,\bar z)$.
Work over $\ind\Sigma_n+\exp$.
Given names $t, \bar u$ and a condition $s\in \Cond$ such that $s\forcesp t\defd$, $s\forcesp \bar u\defd$, 
and $s \forces \fale x{t}{\ex y {\theta(x,y,\bar u)}}$, it is enough to show that 
$s\forces \ex {y^{*}}{\fale x{t}{\exle y{y^{*}} {\theta(x,y,\bar u)}}}$.

For the sake of contradiction, assume that $s\not \forces \ex {y^{*}}{\fale x{t}{\exle y{y^{*}} {\theta(x,y,\bar u)}}}$.
In what follows, we will construct a condition $s'\extendseq s$ and a name $e$ so that $s'\forcesp e\defd$, $s'\forces e\leq t$ and $s'\forces \neg\ex{y}{\theta(e,y,\bar u)}$,
which will lead to a contradiction with $s \forces \fale x{t}{\ex y {\theta(x,y,\bar u)}}$.
We first make the following claim.
\begin{claim*}
For any $a$, there exist $p \ge a$ and $q$ such that
\begin{align*}
(\dag)\, p\in\Wx{s}\wedge q\le \val {t}{p}\wedge \fale y a {\neg\theta(q,y,\val{\bar u}p)}. 
\end{align*}
\end{claim*}
To prove the claim, argue as follows. By our assumptions on $s$, we have
$s \not \forces {\fale x{t}{\exle y{\cname a} {\theta(x,y,\bar u)}}}$ for any $a$, where $\cname a$ is defined as above.
The formula $\fale x{t}{\exle y{\cname a} {\theta(x,y,\bar u)}}$ is either a $\Pi_m$ formula outright if $m = 0$,
or equivalent to a $\Pi_{m}$~formula over $\bd\Pi_{m-1}$ otherwise, 
and in the latter case we know that $s\forces \bd\Pi_{m-1}$ by the proof from the induction hypothesis.
So, we may apply Lemma~\ref{prop:los} to this formula, deriving the existence of arbitrarily large
$p$ for which there exists $q \le t(p)$ such that $(\dag)$ holds. To conclude the proof of the claim,
take some such $p \ge a$.

Notice that $(\dag)$ is a $\Sigma_{n+1}$ statement. We now reason somewhat 
like in the proof of Lemma \ref{lem:ultrapower-fint}. 
Let $H$ be the set of all finite sequences of the form $(p_i,q_i,w_i)_{i \le k}$ such that 
for each $i \le k$, the triple $\langle p_i, q_i, w_i \rangle$ is the smallest one satisfying:
\begin{itemize}
\item $w_i$ is a witness that $p_i,q_i$ satisfy $(\dag)$ with $a=\langle p_{i-1}, q_{i-1}, w_{i-1} \rangle$, 
or with $a =0$ if $i = 0$,
\item $p_i > \langle p_j, q_j, w_j \rangle$ for each $j < i$.
\end{itemize}
Then $H$ is $\Sigma_{0}(\Sigma_{n})$-definable; in fact, in the presence
of $\bd\Sigma_n$ it can be defined by the conjunction of a $\Pi_n$ and a $\Sigma_n$ formula. 
Moreover, $H$ has no maximal element by the Claim, thus it is cofinal.
Define $s'$ so that $p \in \Wx{s'}$ iff there is a finite sequence $(p_i,q_i,w_i)_{i \le k}$ in $H$ such that $p = p_k$. 
Define the $\Sigma_{n+1}$ formula $e$ so that $\Sigma_{n+1}\hyp\Sat(e,\tuple{p,q})$ holds iff 
there is a finite sequence in $H$ such that $p=p_k,q=q_k$.
%
Then $s' \in \Cond$, and clearly $s' \extendseq s$. 
Moreover, by construction we have $s' \forcesp e \defd$ and $s' \forces e \le t$. 
Note in particular that the existence of a number bounding all values of $e$ on arguments
from $\Wx{s'}$ follows from $s \forcesp t \defd$.

It remains to show that $s'\forces \neg\ex{y}{\theta(e,y,\bar u)}$. 
Otherwise, there would be a condition $s''\extendseq s'$ and a name $e'$ 
such that $s''\forcesp e'\defd$ and $s''\forces \theta(e,e',\bar u)$.
By Lemma~\ref{prop:los}, we have
\[\aallin p{\Wx {s''}}{\theta(\val e p,\val{e'} p,\val{\bar u}p)}.\]
Since $s''\forcesp e' \defd$,
there is some number $d$ such that $\forall p {\in}{\Wx {s''}}\,({\val {e'} p \le d})$, 
and thus
\[\aallin p{\Wx {s''}}{\exle y d{\theta(\val e p,y,\val{\bar u}p)}}.\]
Taking large enough $(p_i,q_i,w_i)_{i \le k}$ in $H$ such that $\langle p_{k-1}, q_{k-1}, w_{k-1} \rangle \ge d$ and $p_{k}\in \Wx{s''}$, 
we get $\exle y d{\theta(q_k, y,\val{\bar u}{p_{k}})}$, but this contradicts the definition of $p_{k}$ and $q_{k}$.
\end{proof}

\begin{remark}
For $n > 0$, there is an alternative polynomially $\Pi_{n+2}$-reflecting forcing interpretation
of $\bd\Sigma_{n+1}$ in $\ind\Sigma_n$ 
based on a generic cut construction similar to the one from Section~\ref{sec:ramsey}.
We work in a version $\ind\Sigma_n + (\I)$ in which the axiom $(\I3)$ is
changed to ``every infinite $\Sigma_{n-1}$-definable set contains an $\omega^x$-large
subset for some $x > \I$''. (The axioms $(\I1)$--$(\I3)$ can then be eliminated
by some variant of Lemma~\ref{lem:case-distinction} as in Section \ref{sec:ramsey}).
The role of infinite sets is now played by infinite $\Sigma_{n-1}$-definable sets.
The forcing conditions are finite sets $s$ (which now include e.g.~all bounded $\Delta_n$-definable sets)
that are $\omega^x$-large for some $x > \I$ and, for $n \ge 2$, have the property
that if $\ell_1 <  \ell_2$ are elements of $s$,
then the Skolem function for the first $\exists$ quantifier of the $\Sigma_{n-1}$ universal formula
takes values below $\ell_2$ on inputs below $\ell_1$.
The rest of the argument is along the lines of Section \ref{sec:ramsey}, except that the
step for $\RT^2_2$ in the proof of Lemma \ref{lem:rca-I-fint-rt22}
is replaced by the considerably easier step for the infinite $\Delta_0$ pigeonhole principle
(which is well known to be equivalent to $\bd\Sigma_2$) relative to $0^{(n-1)}$.

In an analogue of this argument for $n = 0$, the conditions
would be finite sets with more than $\I$ elements
that are at least exponentially far apart from each other.
However, the proof of Lemma \ref{lem:rca-I-poly-fint-rt22} (reflection) no longer works.
The proof of the reflection lemma makes use of the fact
that for any element $k$ there is a condition $s$ such that $\min s \ge k$.
But in a model of $\ind \Delta_0 + \exp$ there might
be an element above which the exponential function
may be iterated only a standard (\emph{a fortiori}, no greater than $\I$)
number of times.
\end{remark}

For each fixed $n \in \IN$, the proof of Theorem~\ref{thm:IBfint} can be formalized in
the theory $\PV$, a fragment of bounded arithmetic corresponding to polynomial-time computation.
As already mentioned, this provides a solution to a problem of Clote~et~al.~\cite{art:formconserv}
for $n \ge 1$ and a partial solution for $n=0$.

\begin{cor}\label{cor:coll-ind}
For each $n \in \IN$, the theory $\ind\Sigma_n+\exp$ polynomially simulates $\bd\Sigma_{n+1}+\exp$
with respect to $\Pi_{n+2}$~sentences.
Moreover, the $\Pi_{n+2}$-conservativity of $\bd\Sigma_{n+1}+\exp$ over $\ind\Sigma_n+\exp$
is provable in $\PV$.
\end{cor}

\begin{proof}
The first part follows directly from Theorem~\ref{thm:IBfint} and Theorem~\ref{thm:FI>nspeedup}.
The second part is obtained by formalizing the proof of Theorem~\ref{thm:IBfint} for a fixed $n \in \IN$
in $\PV$.
\end{proof}

In fact, the proof of Theorem~\ref{thm:IBfint} is witnessed by a single polynomial-time algorithm that
takes $n$ (in unary) as input, and this algorithm can also be formalized in $\PV$. As a result,
we can conclude that the $\Pi_{n+2}$-conservativity of $\bd\Sigma_{n+1}+\exp$ over $\ind\Sigma_n+\exp$
holds provably in $\PV$ for any $n$ for which the axioms of $\bd\Sigma_{n+1}+\exp$ over $\ind\Sigma_n+\exp$
actually exist as finite strings --- that is, for any $n$ in the definable cut $\mathrm{Log}$, the domain of the exponential function.

\begin{cor}\label{cor:coll-ind-O1}
$\PV$ proves that for every~$n \in \mathrm{Log}$, the theory $\bd\Sigma_{n+1}+\exp$
is $\Pi_{n+2}$-conservative over $\ind\Sigma_n+\exp$.
\end{cor}

As the reader may have noticed,
 one can actually extract from the truncated generic $\Sigma_{n+1}$ 
 ultrapower construction
   (alternatively, from the generic cut construction for $n>0$)
  a polynomial forcing interpretation of $\WKL_0^*+\bd\Sigma^{0}_{n+1}$ in $\ind\Sigma_n+\exp$
   that is polynomially $\Pi_{n+2}$-reflecting.

\begin{thm}\label{thm:IBWKL}
For each $n \in \IN$,
the theory $\ind\Sigma_n+\exp$ polynomially simulates $\WKL_0^*+\bd\Sigma^0_{n+1}$
with respect to $\Pi_{n+2}$~sentences.
\end{thm}

\begin{proof}
First, notice there exists
 a polynomial forcing interpretation of $\WKL_0^*+\bd\Sigma^0_{n+1}$
  in $\bd\Sigma_{n+1}+\exp$ that polynomially reflects $\Lone$~sentences.
For $n\geq1$, this follows from
 Theorem~3.13(a) in H\'ajek~\cite{incoll:hajek/interpret}.
If $n=0$, then such a forcing interpretation
 can be extracted from Sections~5--8 in Avigad~\cite{art:formforce}. 
Then compose with $\tau(n)$ and invoke Theorem~\ref{thm:FI>nspeedup}.
\end{proof}

\section*{Acknowledgements}
Ko{\l}odziejczyk was partially supported by grant 2017/27/B/ST1/01951 of the National
Science Centre, Poland.
Wong was financially supported by the Institute of Mathematics
  of the Polish Academy of Sciences and 
 then by the Singapore Ministry of Education
  Academic Research Fund Tier~2 grant
   MOE2016-T2-1-019 / R146-000-234-112
    when this research was carried out and
    when this paper was written up.
Yokoyama was partially supported by
JSPS KAKENHI (grant numbers 19K03601 and 15H03634), JSPS
Core-to-Core Program (A.~Advanced Research Networks),
 and JAIST Research Grant 2019 (Houga).
Wong and Yokoyama acknowledge the support of JSPS--NUS grants
 R146-000-192-133 and R146-000-192-733
  during the course of this work.

\bibliographystyle{plain}
\bibliography{arith}

\end{document}